\documentclass[11pt]{amsart}

\usepackage{amssymb}
\usepackage{amsmath, amsfonts,enumerate}
\usepackage{hyperref}
\usepackage{graphics}
\usepackage{amsthm}
\usepackage{latexsym,bm}
\usepackage{euscript}

\let\cal\mathcal

\makeatletter \@mparswitchfalse \makeatother
\newtheorem{theorem}{Theorem}
\newtheorem{lemma}[theorem]{Lemma}

\newtheorem{corollary}[theorem]{Corollary}
\newtheorem{proposition}[theorem]{Proposition}
\theoremstyle{remark}
\newtheorem{remark}[theorem]{Remark}

\theoremstyle{definition}
\newtheorem{definition}[theorem]{Definition}

\theoremstyle{remark}
\newtheorem{example}[theorem]{Example}

\numberwithin{equation}{section}
\numberwithin{theorem}{section}

\def\C{\mathbb C}
\def\Z{\mathbb Z}

\def\R{\mathbb R}
\def\M{\cal{M}}
\def\H{\cal{H}}
\let\<\langle
\def\ch{\raise 0.5ex \hbox{$\chi$}}
\def\T{\tau}
\def\E{\cal{E}}

\let\\\cr
\let\union\bigcup

\let\epsilon\varepsilon

\def\Re{\operatorname{Re}}

\renewcommand{\a}{\alpha}
\renewcommand{\b}{\beta}
\newcommand{\g}{\gamma}

\renewcommand{\l}{\lambda}

\newcommand{\s}{\sigma}

\newcommand{\N}{\cal{N}}
\newcommand{\h}{\mathsf{h}}

\newcommand{\alg}{\mathrm{aa}}

\newcommand{\bmo}{\mathsf{bmo}}
\newcommand{\BMO}{{\mathcal {BMO}}}

\begin{document}

\title[Interpolation]{Interpolation between  noncommutative martingale Hardy  and BMO spaces: the case $\mathbf{0<p<1}$}

\author[Randrianantoanina]{Narcisse Randrianantoanina}
\address{Department of Mathematics, Miami University, Oxford,
Ohio 45056, USA}
 \email{randrin@miamioh.edu}

\date{\today}


\begin{abstract}  Let $\M$ be a semifinite von Nemann algebra equipped with an increasing filtration $(\M_n)_{n\geq 1}$  of (semifinite) von Neumann subalgebras  of $\M$.  For
  $0<p <\infty$,  let $\h_p^c(\M)$  denote the  noncommutative  column  conditioned martingale Hardy space  and $\bmo^c(\M)$ denote the column \lq\lq little\rq\rq  \  martingale BMO space associated with  the filtration $(\M_n)_{n\geq 1}$.

We prove the following real interpolation identity:
 if $0<p <\infty$ and   $0<\theta<1$, then  for   $1/r=(1-\theta)/p$, 
  \[
  \big(\h_p^c(\M), \bmo^c(\M)\big)_{\theta, r}=\h_{r}^c(\M),
  \]
  with equivalent quasi norms.

For the case of complex interpolation, we obtain that if $0<p<q<\infty$  and $0<\theta<1$, then  for $1/r =(1-\theta)/p +\theta/q$,
\[
 \big[\h_p^c(\M), \h_q^c(\M)\big]_{\theta}=\h_{r}^c(\M)
\]
with equivalent quasi norms. 

These  extend  previously known results from  $p\geq 1$ to the full range $0<p<\infty$. Other related spaces   such as spaces of adapted sequences  and Junge's noncommutative conditioned $L_p$-spaces are also  shown to form interpolation scale for the  full range $0<p<\infty$ when either the real method or the complex  method is used.  Our method of proof is based on  a new algebraic atomic decomposition for Orlicz space version of 
Junge's noncommutative conditioned $L_p$-spaces. 

We apply these results to derive various inequalities for martingales in noncommutative  symmetric  quasi-Banach spaces.

\end{abstract}

\maketitle

\section{Introduction}

Hardy space theory takes on many forms and appears in many aspects of mathematics such as harmonic analysis, PDE's, functional analysis, probability theory, and many others. Interpolation spaces  between Hardy spaces in various contexts have a long history.  We refer to the articles \cite{Fefferman-Stein, Jones, Pisier-int} for some background on
interpolations between classical Hardy spaces from harmonic analysis and \cite{Jason-Jones,  Weisz} for  interpolations between Hardy spaces from martingale theory.  On the other hand, the theory of noncommutative martingales has  seen rapid development in many directions. Indeed,  since the establishment of the noncommutative Burkholder-Gundy inequalities in \cite{PX}, many classical inequalities are now understood for this context.  We  refer to the book \cite[Chap.~14]{Pisier-Martingale} for  a summary of some of the  main  inequalities from noncommutative martingale theory.
Further references relevant to our purpose are  \cite{Bekjan-Chen-Perrin-Y,  Chen-Ran-Xu, Ju, JX, Musat-inter, Ran-Wu-Zhou}.
The main focus of the present article is on Hardy spaces arising from noncommutative martingale theory. More specifically,  column/row Hardy  spaces defined from column/row conditioned square functions initiated by Junge and Xu   in \cite{JX} in connection with the noncommutative Burkholder/Rosenthal inequalities. These spaces are generally  referred to  as conditioned Hardy spaces and its column (resp.  row) version  is usually  denoted by $\h_p^c$ (resp.  $\h_p^r$). We would like to  emphasize that  this  particular class  of martingale Hardy spaces  is  instrumental in   classical theory. We refer to  the monograph \cite{Weisz} for more in depth treatment of the classical setting. Likewise,   the Hardy spaces $\h_p^c$ and $\h_p^r$ are  proven to be fundamental objects  in  various aspects of the new progress made in the noncommutative martingale theory during  the last several years. For instance, the formulation of the  noncommutative Burkholder inequality for the case $1<p<2$  was mainly  due to  a reformulation of the case $2\leq p<\infty$ as equivalence of norms involving  the spaces $\h_p^c$ and $\h_p^r$. The asymmetric Doob maximal inequalities in \cite{Hong-Junge-Parcet}  were derived from properties of $\h_p^c$ and $\h_p^r$ where $1\leq p<2$.
The general theme of the present article is on interpolation spaces  for the quasi-Banach  space couple   $(\h_p^c, \h_q^c)$ when  $0<p<q<\infty$.

To motivate our consideration, let us review some interpolation results from classical martingale theory. Let $(\Omega, \cal{F},\mathbb{P})$ be a probability space and $(\cal{F}_n)_{n\geq 1}$ be  an increasing sequence of $\sigma$-subfields of $\cal{F}$  satisfying  the condition $\cal{F}=\sigma(\union_{n\geq 1} \cal{F}_n)$.
For $0<p \leq \infty$, denote by $\H_p(\Omega)$ (resp. $\h_p(\Omega)$) the martingale Hardy space  defined by square functions (resp. conditioned square functions)  and $\BMO(\Omega)$  (resp. $\bmo(\Omega)$) the martingale BMO  (resp. martingale little BMO) space associated with the filtration $(\cal{F}_n)_{n\geq 1}$. We refer to \cite{GA,Weisz} for precise definitions  and properties of these spaces along with discussions on their importance for the classical theory. It is well established in the literature that the spaces $\BMO(\Omega)$ and $\bmo(\Omega)$ play important  role in interpolation
theory  as  they may be used as  natural substitutes for $L_\infty(\Omega)$.
Our motivation comes from  the following  three classical   results  that involved  martingale BMO spaces as one of the endpoints of interpolations:
 
\begin{equation}\label{classical-BMO}
\big[\H_1(\Omega) , \BMO(\Omega)\big]_{\theta} = \H_r(\Omega), \quad 0<\theta<1  \ \text{and}\  \frac{1}{r}= 1-\theta;
\end{equation}

\begin{equation}\label{H-BMO}
\big[\H_p(\Omega) , \BMO(\Omega) \big]_{\theta} \subsetneq \H_r(\Omega), \quad  0<p<1,\ 0<\theta<1,\  \text{and} \  \frac{1}{r}= \frac{1-\theta}{p}<1;
\end{equation}

\begin{equation}\label{h-BMO}
\big(\h_p(\Omega) , \bmo(\Omega) \big)_{\theta, r} = \h_r(\Omega), \quad   0<p<\infty, \  0<\theta<1, \ \text{and}\  \frac{1}{r}= \frac{1-\theta}{p},
\end{equation}
where $[\cdot,\cdot]_\theta$ (resp. $(\cdot,\cdot)_{\theta, r}$) denotes  the complex (resp. real) interpolation method.  The identity \eqref{classical-BMO}  and the  inclusion  \eqref{H-BMO} were   obtained by Janson and Jones in \cite{Jason-Jones} while \eqref{h-BMO}  was established in its present form  by Weisz in \cite{Weisz3}.
It is a natural question to consider if the three assertions stated above still  hold for the noncommutative setting. The first result in this direction is due to Musat in  \cite{Musat-inter} who proved  a  noncommutative analogue of \eqref{classical-BMO} where as introduced  in \cite{PX}, the noncommutative Hardy space $\H_1$ is the sum  of the column version and the row version and the BMO space is defined as the intersection of column BMO and the row BMO. Later, Bekjan  {\it et al.} established in \cite{Bekjan-Chen-Perrin-Y} that the noncommutative analogue of \eqref{h-BMO} holds  for the Banach space range. That is,   \eqref{h-BMO} remains valid for the interpolation couple $(\h_1^c, \bmo^c)$. To the best of our knowledge,  \cite{Bekjan-Chen-Perrin-Y} and \cite{Musat-inter}  are the only articles available in the literature that contain substantial  advances in the study of interpolation spaces of  noncommutative martingale Hardy spaces to date. 
The present paper extends the interpolation results from \cite{Bekjan-Chen-Perrin-Y}  to the case $0<p<1$ thus providing a full  noncommutative generalization of \eqref{h-BMO}  for column/row spaces. We refer to  Theorem~\ref{main-interpolation} below for  detailed  formulation.  We also obtained  interpolation spaces between  spaces of adapted sequences which may be viewed as the right substitute for \eqref{classical-BMO}  when $0<p<1$. This is stated in Theorem~\ref{interpolation-adapted}. We should point out  that the  result  on the spaces of adapted sequences appears to be new even for the classical setting. Moreover,  the noncommutative analogue of \eqref{H-BMO} can be easily deduced from the  result on adapted sequences.

Our method of proof is very different from \cite{Bekjan-Chen-Perrin-Y, Musat-inter}. In fact, the techniques used in  both \cite{Bekjan-Chen-Perrin-Y} and \cite{ Musat-inter} appear to work only for Banach couples.  Our main objective   is  to  compare  $K$-functionals for the couple $(\h_p^c, \h_q^c)$ where $0<p<q<\infty$ to those associated to the well-known couple $(L_p, L_q)$  associated with appropriate   amplified  von Neumann algebras. In the classical setting,  this  type of reduction is usually achieved  through some strategic use  of stopped martingales. We note  that  a the time of this writing there is no  direct analogue of stopping times  available for the noncommutative setting but  the so-called Cuculescu projections (\cite{Cuc}) are often used as a substitute for stopping times.  However, 
Cuculescu projections do not appear to be efficient enough to provide the desired truncations. This makes our approach very 
 different  from the classical   setting as found in  \cite{Jason-Jones, Weisz3}.  Our method  is based on the so-called algebraic atomic decompositions. The notion of algebraic atoms were introduced  by Perrin in  \cite{Perrin2} for Junge's  noncommutative  conditioned $L_p$ spaces 
and the Hardy spaces $\h_p^c$ for $1\leq p<2$ and found to be instrumental in the study of Doob's maximal inequality for martingale in  noncommutative Hardy spaces (\cite{Hong-Junge-Parcet}). Algebraic atomic decompositions for   the Hardy space $\h_p^c$  with $0<p<1$  was recently studied in \cite{Chen-Ran-Xu} by  constructive approach.
Another recent development in this direction is that the noncommutative  Orlicz-Hardy space $\h_\Phi^c$  admits algebraic atomic decomposition for convex function $\Phi$ satisfying some natural conditions. Using insights from  the constructive approach used  in \cite{Chen-Ran-Xu} and \cite{Ran-Wu-Zhou}, we established  that    noncommutative conditioned  Orlicz-Hardy spaces
admit  version of algebraic  atomic decompositions  when the Orlicz function is $p$-convex and $q$-concave for $0<p<q<2$ (see Theorem~\ref{main-algebraic}   for details). These  more general atomic decompositions  turn out to be one of  the decisive tools we used  for our results on $K$-functionals for  the couple $(\h_p^c, \h_q^c)$ when the distance between $p$ and $q$ is sufficiently small.   The results on $K$-functionals coupled with the well-known  Wolff's  interpolation theorem  provide the full noncommutative generalization of \eqref{h-BMO}. The case of  spaces of adapted sequences  is handle with the same techniques. That is, estimating the $K$-functionals through some version of algebraic atomic decompositions.

\medskip

The paper is organized as follows. In the next section, we review the basics of 
  noncommutative symmetric  quasi-Banach spaces following the formulation of \cite{Kalton-Sukochev, X} and give full detailed accounts of  the construction of  noncommutative  martingale conditioned Hardy  spaces associated with symmetric quasi-Banach spaces.  We also formulate and prove the algebraic  decomposition for  noncommutative conditioned spaces associated with Orlicz function spaces and for spaces of adapted sequences. 
     
  Section~3 contains  formulations and proofs of our principal results: noncommutative generalization of \eqref{H-BMO}, noncommutative generalization of \eqref{h-BMO}, and an extension of \eqref{classical-BMO}  to  the full range when spaces of adapted sequences are used.  
  
  In Section~4,  we explore  some further applications of  our results  and methods from  Section~2 and Section~3 to various inequalities involving  noncommutative martingale Hardy spaces  in the context of general symmetric spaces of measurable operators.


\section{Definitions and preliminary results}
Throughout, we adopt the notation  $A\lesssim_\alpha B$ to indicate that there is a constant $C_\alpha$ depending only on the parameter $\alpha$ such that  the inequality $A \leq C_\alpha B$ is satisfied. Similarly, $A\approx_\alpha B$ is used  if both $A\lesssim_\alpha B$ and $B \lesssim_\alpha A$ hold.

\subsection{Noncommutative symmetric spaces}
Throughout,   $\M$  denotes  a semifinite von Neumann algebra equipped with a faithful normal semifinite trace $\T$.  Let  $\widetilde{\M}$ denote  the associated topological $*$-algebra of $\T$-measurable operators in the sense of  \cite{FK}.  For $x \in \widetilde{\M}$, we recall   that its  \emph{generalized singular number} $\mu(x)$ 
is  the real-valued function defined by
\[
\mu_t(x)= \inf\big\{ s>0 : \T\big(\ch_{(s,\infty)}(|x|)\big)\leq t\big\}, \quad t>0,
\]
where $\ch_{(s,\infty)}(|x|)$ is the spectral projection of $|x|$ associated with the interval $(s,\infty)$.
We observe that
 if $\M$ is the abelian von Neumann algebra $L_\infty(0,\infty)$ with  the trace given by  integration  with respect to  the Lebesgue  measure, then $\widetilde{\M}$  becomes  the space of    measurable  complex functions  on $(0,\infty)$ which are bounded except on a set of finite measure and for  $f\in \widetilde{\M}$,  $\mu(f)$ is precisely  the usual decreasing rearrangement of  the function  $|f|$. We refer to \cite{PX3} for more information on noncommutative integration. 

We denote by  $L_0$,  the space of measurable functions on the interval $(0,\infty)$. 
Recall that a  quasi-Banach function space  $(E,\|\cdot\|_E)$ of measurable functions  on the interval $(0,\infty)$ is called \emph{symmetric} if for any $g \in E$ and any $f \in L_0$ with $\mu(f) \leq \mu(g)$, we have $f \in E$ and $\|f\|_E \leq \|g\|_E$. Throughout, all function spaces are assumed to be defined on the interval $(0,\infty)$.

Let  $E$ be a symmetric quasi-Banach function space. We  define the  corresponding  noncommutative space by setting:
\begin{equation*}
E(\M, \T) = \Big\{ x \in
\widetilde{M}\ : \ \mu(x) \in E \Big\}. 
\end{equation*}
Equipped with the  quasi-norm
$\|x\|_{E(\M,\T)} := \| \mu(x)\|_E$,   the linear space $E(\M,\T)$ becomes a complex quasi-Banach space (\cite{Kalton-Sukochev,X}) and is usually referred to as the \emph{noncommutative symmetric space} associated with $(\M,\T)$ corresponding to  $E$. 
 We remark that  if $0< p<\infty$ and $E=L_p$, then $E(\M, \T)$ is exactly   the usual noncommutative $L_p$-space  $L_p(\M,\T)$ associated with  $(\M,\T)$. In the sequel, $E(\M,\T)$ will be abbreviated to $E(\M)$. 
 
 Other  classes of examples that are relevant for our purpose are the class of Orlicz spaces and the class of Lorentz spaces. We review these two classes for convenience.
 
 A function $\Phi: [0,\infty) \to [0,\infty)$ is  called an \emph{Orlicz  function} whenever it is strictly increasing, continuous,  $\Phi(0)=0$, and $\lim_{u\to \infty}\Phi(u)=\infty$.
 The \emph{Orlicz space} $L_\Phi$ is the collection of all  $f\in L_0$ for which 
   there exists a constant $c$  such that   $I_\Phi(|f|/c)<\infty$  where the modular functional $I_\Phi(\cdot)$ is defined by:
 \[
 I_\Phi(|g|)=\int_0^\infty \Phi(|g(t)|)\ dt, \quad g \in L_0.
 \]
The space $L_\Phi$ is equipped  with the Luxemburg  quasi-norm:
\[
\big\|f \big\|_{L_\Phi}= \inf\big\{ c>0 : I_\Phi(|f|/c) \leq 1\big\}.
\] 
If $\Phi$ is convex then $L_\Phi$ is a symmetric Banach function space. However, we do not restrict ourselves to just the case of normed spaces. 
We refer to \cite{Turpin1} for Orlicz spaces associated  with non necessarily convex functions. 

We recall that for $0<p \leq q<\infty$, $\Phi$ is called  \emph{$p$-convex} (resp., \emph{$q$-concave}) if  the function $t\mapsto \Phi(t^{1/p})$ (resp., $t\mapsto \Phi(t^{1/q})$) is convex (resp., concave).
Below, we only consider  Orlicz spaces associated  with Orlicz  functions  that are  $p$-convex and $q$-concave for some $0<p\leq q<2$.
 
 \medskip
 
For the next relevant example,  assume that  $0<p,q\leq \infty$. The \emph{Lorentz space} $L_{p,q}$ is the space  of all $f \in L_0$ for which $\|f\|_{p,q}<\infty$ where
 \begin{equation*}
 \big\|f \big\|_{p,q}  =\begin{cases}
 \left(\displaystyle{\int_{0}^\infty \mu_{t}^{q}(f)\
d(t^{q/p})}\right)^{1/q},  &0< q < \infty; \\
\displaystyle{\sup_{t >0} t^{1/p} \mu_t(f)}, &q= \infty.
\end{cases} 
 \end{equation*}
If $1\leq q\leq p <\infty$ or $p=q=\infty$, then $L_{p,q}$ is a symmetric Banach function space.  If $1<p<\infty$ and $p\leq q\leq \infty$, then $L_{p,q}$ can be equivalently renormed to become a symmetric Banach function (\cite[Theorem~4.6]{BENSHA}). In general, $L_{p,q}$ is only a symmetric quasi-Banach function space.

 Both  noncommutative Orlicz spaces and  noncommutative Lorentz spaces will be heavily involved throughout the paper.
 
\subsection{Martingale Hardy spaces and conditioned spaces}

In the sequel, we always denote by $(\M_n)_{n \geq 1}$ an
increasing sequence of von Neumann subalgebras of ${\M}$
whose union  is w*-dense in
$\M$. For  every $n\geq 1$,  we assume that there is a trace preserving conditional expectation $\E_n$
from ${\M}$ onto  ${\M}_n$.

\begin{definition}
A sequence $x = (x_n)_{n\geq 1}$ in $L_1(\M)+\M$ is called \emph{a
noncommutative martingale} with respect to $({\M}_n)_{n \geq
1}$ if $\mathcal{E}_n (x_{n+1}) = x_n$ for every $n \geq 1$.
\end{definition}

Let $E$ be a symmetric quasi-Banach function space  and $x=(x_n)_{n\geq 1}$ be a martingale. If  for every $n\geq 1$, $x_n \in E(\M_n)$, then we  say that  $(x_n)_{n\geq 1}$ is an $E(\M)$-martingale.  In this case, we set
\begin{equation*}\| x \|_{E(\M)}= \sup_{n \geq 1} \|
x_n \|_{E(\M)}.
\end{equation*}
If $\| x \|_{E(\M)} < \infty$, then $x$   will be called
a bounded $E(\M)$-martingale.

For a martingale $x=(x_n)_{n\geq 1}$, we set $dx_n=x_n-x_{n-1}$  for $n\geq 1$ with the usual convention that $x_0=0$. The sequence $dx=(dx_n)_{n\geq 1}$ is called the \emph{martingale difference sequence} of $x$. A martingale $x$  is called a \emph{finite martingale} if there exists $N$ such that $dx_n=0$ for all $n\geq N$.

Let us now  review some basic  definitions related to   martingale Hardy  spaces associated to   noncommutative   symmetric spaces. 

Following \cite{PX}, we define  the  \emph{column square functions} of a given
  martingale $x = (x_k)$ by setting:
 \[
 S_{c,n} (x) = \Big ( \sum^n_{k = 1} |dx_k |^2 \Big )^{1/2}, \quad
 S_c (x) = \Big ( \sum^{\infty}_{k = 1} |dx_k |^2 \Big )^{1/2}\,.
 \]
The conditioned versions were introduced in \cite{JX}. For  a given $L_2(\M) +\M$-martingale $(x_k)_{k\geq 1}$, we set 
  \[
 s_{c,n} (x) = \Big ( \sum^n_{k = 1} \E_{k-1}|dx_k |^2 \Big )^{1/2}, \quad
 s_c (x) = \Big ( \sum^{\infty}_{k = 1} \E_{k-1}|dx_k |^2 \Big )^{1/2}\,.
 \]
 The operator $s_c(x)$ is  called  the \emph{column conditioned square function} of $x$.
  For convenience, we will use the notation 
  \[
  \cal{S}_{c,n}(a)= \Big ( \sum^n_{k = 1} |a_k |^2 \Big )^{1/2}, \quad
 \cal{S}_c (a) = \Big ( \sum^{\infty}_{k = 1} |a_k |^2 \Big )^{1/2}\,.
 \]
and  
 \[
  \s_{c,n} (b) = \Big ( \sum^{n}_{k = 1} \E_{k-1}|b_k |^2 \Big )^{1/2}, \quad
 \s_c (b) = \Big ( \sum^{\infty}_{k = 1} \E_{k-1}|b_k |^2 \Big )^{1/2}
 \]
 for sequences $a=(a_k)_{k\geq 1}$ in $L_1(\M)+\M$  and $b=(b_k)_{k\geq 1}$ in $L_2(\M) +\M$   that are not necessarily  martingale difference sequences.  
 It is worth pointing out that the infinite sums of positive operators stated above may not always make sense as operators  but we only consider below special cases where they do converge in the sense of  the measure topology.
 
 We will now  describe various  noncommutative martingale Hardy spaces associated with symmetric quasi-Banach  function spaces.
 
Assume  that $E$ is a symmetric  quasi-Banach function space. We denote by $\cal{F}_E$ the collection  of all finite martingales in $E(\M)$.  
For  $x=(x_k)_{k\geq 1} \in \cal{F}_E$, we set:
\[
\| x \|_{\mathcal{H}_E^c}= \| \cal{S}_c(x)\|_{E(\M)}\,. \]
Then $(\cal{F}_E, \|\cdot\|_{\cal{H}_E^c})$ is a quasi-normed space.
 If we denote by $(e_{i,j})_{i,j \geq 1}$   the family of unit matrices in $B(\ell_2(\mathbb{N}))$, then 
 the  correspondence  $x\mapsto \sum_{k\geq 1} dx_k \otimes e_{k,1}$ maps $\cal{F}_E$ isometrically into 
a (not necessarily closed) linear subspace of $E(\M\overline{\otimes} B(\ell_2(\mathbb{N})))$.
We define the  column Hardy space
  $\mathcal{H}_E^c (\mathcal{M})$ to  be the completion  of $(\cal{F}_E, \|\cdot\|_{\cal{H}_E^c})$.   It then follows that 
$\H_E^c(\M)$  embeds isometrically into a closed subspace of  the quasi-Banach space $E(\M\overline{\otimes} B(\ell_2(\mathbb{N})))$.

 In the sequel, we will also make use of the  more general space $E(\M;\ell_2^c)$ which is defined as the set of all sequences  $a=(a_k)$ in $E(\M)$ for which $\cal{S}_c(a) \in E(\M)$.  In this case, we set
\[
\big\|a \big\|_{E(\M;\ell_2^c)} = \|\cal{S}_c(a)  \|_{E(\M)}.
\]
 Under the above quasi-norm, one can easily see that  $E(\M;\ell_2^c)$ is a quasi-Banach space. 
  The closed  subspace of $E(\M;\ell_2^c)$  consisting of adapted sequences will be denoted by $E^{\rm ad}(\M;\ell_2^c)$. That is,
\[
E^{\rm ad}(\M;\ell_2^c)=\Big\{  (a_n)_{n\geq 1} \in E(\M;\ell_2^c) :  \forall n\geq 1, a_n \in E(\M_n) \Big\}.
\]
Note that for  $1<p<\infty$,  it follows from the noncommutative  Stein inequality that $L_p^{\rm ad}(\M;\ell_2^c)$ is a complemented subspace of  $L_p(\M;\ell_2^c)$.  One should not expect  such complementation if one merely assume that $E$ is quasi-Banach symmetric  function space.


\bigskip

Next, we will discuss conditioned versions of the  spaces defined earlier.
Consider the linear space $\cal{FS}$ consisting  of all $x \in \M$ such that there exists a projection $e\in \M_1$, $\T(e)<\infty$, and $x=exe$. We should note that if $\M$ is finite,  then $\cal{FS}=\M$. 
 For every $n\geq 1$ and $0<p \leq \infty$,  we define the  space $L_p^c(\M,\E_n)$ to be the completion of $\cal{FS}$ with respect to the quasi-norm:
 \[
 \big\|x\big\|_{L_p^c(\M,\E_n)} =\big\|\E_n(x^*x) \big\|_{p/2}^{1/2}.
 \]
 We would like to emphasize  here that if $x=exe \in\cal{FS}$  is as described above,  then $\E_n(x^*x)=e\E_n(x^*x)e$ is a well-defined operator in $\M$ and since $\T(e)<\infty$,  it follows that $\E_n(x^*x) \in L_{p/2}(\M)$.
 
 According to \cite{Ju},  for every $0<p\leq \infty$, there exists an isometric right $\M_n$-module map  $u_{n,p}: L_p^c(\M, \E_n) \to L_p(\M_n;\ell_2^c)$ such that
\begin{equation}\label{u}
u_{n,p}(x)^* u_{n,q}(y)=\E_n(x^*y) \otimes e_{1,1},
\end{equation}
whenever $x\in L_p^c(\M;\E_n)$, $y \in L_q^c(\M;\E_n)$, and  $1/p +1/q \leq 1$. 
An important fact about these maps is that they are independent of $p$ as the index $p$ in the presentation of \cite{Ju} was only needed to accommodate  the non-tracial case. Below, we will simply use $u_n$ for $u_{n,p}$.

For $0<p\leq \infty$, the range of $u_n$ is complemented in $L_p(\M_n,\ell_2^c)$. In fact,  as proved in  \cite[Proposition~2.8(iii)]{Ju},
 there exists a contractive projection $\cal{Q}_{n}$ from  $L_p(\M_n ; \ell_2^c)$ onto the range of  $u_{n}$ such that for every $\xi \in L_p(\M_n;\ell_2^c)$,
 \[
 \cal{Q}_{n}(\xi)^*\cal{Q}_{n}(\xi)\leq  \xi^*\xi.
 \]
This fact will be used in the sequel.

Let  $\mathfrak{F}$  be the collection of all finite sequences $a=(a_n)_{n\geq 1}$ in $\mathcal{FS}$. For $0<p<\infty$, we  defined the \emph{conditioned space}
  $L_p^{\rm cond}(\M; \ell_2^c)$ to be   the completion  of  the linear space $\mathfrak{F}$  with respect to the quasi-norm:
\begin{equation}\label{conditioned-norm}
\big\| a \big\|_{L_p^{\rm cond}(\M; \ell_2^c)} = \big\| \sigma_c(a)\big\|_p
\end{equation}
(here, we take $\E_0=\E_1$).
According to \cite{Ju},  $L_p^{\rm cond}(\M;\ell_2^c)$ can be isometrically embedded into an $L_p$-space associated to  a semifinite von  Neumann algebra  by means of the following map: 
\[
U : L_p^{\rm cond}(\M; \ell_2^c) \to L_p(\M \overline{\otimes} B(\ell_2(\mathbb{N}^2)))\]
defined by setting: 
\[
U((a_n)_{n\geq 1}) = \sum_{n\geq 1} u_{n-1}(a_n) \otimes e_{n,1}, \quad (a_n)_{n\geq 1} \in \mathfrak{F}.
\]
The range  of $U$ may be viewed as a double indexed sequences $(x_{n,k})$  such that $x_{n,k} \in L_p(\M_n)$ for all $k\geq 1$. As  an operator affiliated with $\M \overline{\otimes} B(\ell_2(\mathbb{N}^2))$,  this may be expressed as 
$\sum_{n,k} x_{n,k} \otimes e_{k,1} \otimes e_{n,1}$. 
It is immediate   from \eqref{u}  
   that if $(a_n)_{n\geq 1} \in \mathfrak{F}$ and $(b_n)_{n\geq 1} \in \mathfrak{F}$, then
\begin{equation}\label{key-identity}
U( (a_n))^* U((b_n)) =\big(\sum_{n\geq 1} \E_{n-1}(a_n^* b_n)\big) \otimes e_{1,1} \otimes e_{1,1}.
\end{equation}
In particular, if $(a_n)_{n\geq 1}  \in \mathfrak{F}$ then  $\| (a_n) \|_{L_p^{\rm cond}(\M; \ell_2^c)}=\|U( (a_n))\|_p$ and 
hence  $U$ is indeed an isometry.  

 Now, we  generalize  the notion of  conditioned spaces to the setting of symmetric spaces of operators. This is done in steps.
 
 $\bullet$  Assume  first that  $E$ is a symmetric quasi-Banach function space satisfying  $L_p \cap L_\infty \subseteq E\subseteq L_p +L_\infty$ for some $0<p<\infty$ and $L_p  \cap L_\infty$ is dense in $E$.  This is the case for instance when $E$ is a separable fully symmetric quasi-Banach function space.     For a given 
   sequence $a=(a_n)_{n\geq 1} \in \mathfrak{F}$, we set:
\[
\big\| (a_n) \big\|_{E^{\rm cond}(\M; \ell_2^c)} = \big\| \sigma_c(a)\big\|_{E(\M)}=\big\| U( (a_n))\big\|_{E(\M \overline{\otimes} B(\ell_2(\mathbb{N}^2)))}.
\]
This is well-defined and  induces   a  quasi-norm  on the linear space $\mathfrak{F}$.  We define   the quasi-Banach space $E^{\rm cond}(\M;  \ell_2^c)$ to be the completion of  the quasi normed space $(\mathfrak{F}, \|\cdot\|_{E^{\rm cond}(\M; \ell_2^c)} )$.  The space $E^{\rm cond}(\M;\ell_2^c)$ will be called \emph{the  column conditioned space associated with $E$}.
It is clear that  
 $U$ extends to an isometry from  $E^{\rm cond}(\M;\ell_2^c)$ into $E(\M \overline{\otimes} B(\ell_2(\mathbb{N}^2)))$ which we will still denote by $U$. 
 
 Below, we use the notation $E^{0,\rm cond}(\M;  \ell_2^c)$  for  the closure of the linear space
 $\mathfrak{F}^0=\{a=(a_n) \in \mathfrak{F}: a_1=0\}$ in  $E^{\rm cond}(\M;  \ell_2^c)$.  One can easily see that for $a=(a_n)_{n\geq 1} \in \mathfrak{F}$, we have
 \[
 \max\Big\{ \big\| \E_1(|a_1|^2)^{1/2} \big\|_{E(\M_1)},  \big\| (a_n)_{n\geq 2}\big\|_{E^{\rm cond}(\M;\ell_2^c)} \Big\} \leq \big\| (a_n)_{n\geq 1}\big\|_{E^{\rm cond}(\M;\ell_2^c)}. 
 \]
This shows that we have the direct sum
\begin{equation}\label{direct-sum}
E^{\rm cond}(\M;\ell_2^c)= E^c(\M;\E_1) \oplus E^{0,\rm cond}(\M;\ell_2^c).
\end{equation}
 
$\bullet$  Assume now that $E\subseteq L_p +L_q$  for some $0<p,q<\infty$ that is not necessarily separable. We set 
\[
E^{\rm cond}(\M; \ell_2^c)=\Big\{ x \in (L_p +L_q)^{\rm cond}(\M;\ell_2^c) : U(x) \in E(\M \overline{\otimes} B(\ell_2(\mathbb N^2)) )\Big\}
\]
 equipped with the  quasi-norm:
 \[
 \big\|x \big\|_{E^{\rm cond}(\M; \ell_2^c)} =\big\| U(x) \big\|_{E(\M \overline{\otimes} B(\ell_2(\mathbb{N}^2)))}. 
 \]
 It is clear that $E^{\rm cond}(\M; \ell_2^c)$ as defined is  a linear quasi-normed space. We claim that it is complete. Indeed,  if $(x_\nu)_{\nu\geq 1}$ is  a Cauchy
 sequence in $E^{\rm cond}(\M; \ell_2^c)$, then it  converges to some $x \in (L_p +L_q)^{\rm cond}(\M;\ell_2^c)$. Since $( U(x_\nu))_{\nu\geq 1}$ is also a Cauchy sequence in the quasi-Banach space $E(\M \overline{\otimes} B(\ell_2(\mathbb{N}^2)))$ and $E(\M \overline{\otimes} B(\ell_2(\mathbb{N}^2))) \subseteq (L_p +L_q)(\M \overline{\otimes} B(\ell_2(\mathbb{N}^2)))$, it follows that $x \in E^{\rm cond}(\M; \ell_2^c)$ and $(x_\nu)_{\nu\geq 1}$  converges to $x$ in $E^{\rm cond}(\M; \ell_2^c)$. We should note here that if  $L_p \cap L_\infty$ is dense in $E$ then the above definition coincides  with the one described in the previous bullet. We also define
\[
E^{0, \rm cond}(\M; \ell_2^c)=(L_p +L_q)^{0, \rm cond}(\M; \ell_2^c) \cap E^{\rm cond}(\M; \ell_2^c).
\]
The direct sum stated in \eqref{direct-sum} still applies.

\begin{remark}
At the time of this writing, we do not know  of any suitable definition  for conditioned space associated with $L_p  + L_\infty$ when $0<p<2$.  It is also unclear if  our definition of $E^{\rm cond}(\M;\ell_2^c)$ for non separable space $E\subset L_p + L_q$ when $0<p<2$ is independent of the isometry $U$.
\end{remark}
  
\medskip
  
 We now recall the construction of column  conditioned martingale Hardy spaces.
  As in the conditioned spaces,  we describe the noncommutative  conditioned Hardy spaces in steps. Let $\mathfrak{F}(M)$ be the collection of all finite martingale  $(x_n)_{1\leq n \leq N}$   for which  $ x_N  \in \cal{FS}$. We can easily see that for every $0<p\leq \infty$, $\mathfrak{F}(M) \subseteq  \h_p^c(\M)$. As in the case of conditioned spaces, $\mathfrak{F}(M)$ is dense in $\h_p^c(\M)$ when $0<p<\infty$.

$\bullet$  First, assume that $E\subseteq  L_2 + L_\infty$. In this case, column conditioned square functions are well-defined for  bounded martingales in $E(\M)$. We define
$\h_E^c(\M)$ to be the collection of all  bounded martingale $x$ in $E(\M)$ for which $s_c(x) \in E(\M)$. We equip $\h_E^c(\M)$ with the norm:
\[
 \big\|x \big\|_{\h_E^c} = \big\|s_c(x)\big\|_{E(\M)}.
 \]
 One can easily verify that $(\h_E^c(\M), \|\cdot\|_{\h_E^c})$ is complete.

$\bullet$  Next, we consider  quasi-Banach space  $E$  such  that  $L_p\cap L_\infty$ is dense in $E$ for some $0<p<\infty$. This is the case if $E$ is separable.
 Let $x \in \mathfrak{F}(M)$.  As noted above, $s_c(x) \in L_p(\M) \cap \M$. In particular, $s_c(x) \in E(\M)$.   We equip $\mathfrak{F}(M)$  with the quasi-norm
 \[
 \big\|x \big\|_{\h_E^c} = \big\|s_c(x)\big\|_{E(\M)}=\big\|(dx_n)\big\|_{E^{\rm cond}(\M;\ell_2^c)}.
 \]
 The column conditioned Hardy space  $\h_E^c(\M)$ is the completion of $(\mathfrak{F}(M), \|\cdot\|_{\h_E^c})$.   Clearly, the map  $x\mapsto (dx_n)$ (from $\mathfrak{F}(M)$ into 
 $\mathfrak{F}$)  extends to   an isometry from $\h_E^c(\M)$ into $E^{\rm cond}(\M;\ell_2^c)$ which we denote by $\mathcal{D}_c$. In particular, $\h_E^c(\M)$ is isometrically isomorphic to a  subspace of $E(\M\overline{\otimes}B(\ell_2(\mathbb{N}^2)))$ via the isometry $U\cal{D}_c$. We should note here that if $L_2 \cap L_\infty$ is dense in $E$ and $E\subseteq L_2 +L_\infty$, then the two definitions provide the same space.
 
 $\bullet$ Assume now that $E \subset L_q + L_q$ for $0<p,q<\infty$. As in the case of conditioned spaces, we set
 \[
 \h_E^c(\M)=\Big\{ x \in \h_{L_p +L_q}^c(\M) : U\cal{D}_c(x) \in E(\M \overline{\otimes} B(\ell_2(\mathbb N^2)) )\Big\}
 \]
equipped with the quasi-norm: 
 \[
 \big\|x \big\|_{\h_E^c} =\big\| U\cal{D}_c(x) \big\|_{E(\M \overline{\otimes} B(\ell_2(\mathbb{N}^2)))}. 
 \]
Since the operator $U$ is independent of the index, one can easily see that the  space $\h_E^c(\M)$  is independent of $p$ and $q$. Moreover, one can verify as in the case of conditioned spaces that the quasi-normed space  $(\h_E^c(\M), \|\cdot\|_{\h_E^c})$ is complete. Furthermore, if   $E$ is such that $L_2 \cap L_\infty$ is dense in $E$ then $\h_E^c(\M)$ coincides with the one defined through completion considered in the second bullet.

 The following fact will be used in the sequel.
 \begin{lemma}\label{complemented}
 Let  $0<p<q<\infty$ and assume that $E\subseteq L_p +L_q$. 
 The Hardy space $\h_E^c(\M)$ is  $1$-complemented  in $E^{\rm cond}(\M;\ell_2^c)$. More precisely,
 there is an  onto    map $\Pi :   E^{\rm cond}(\M;\ell_2^c) \to \h_E^c(\M)$ with 
 $\|\Pi\|=1$ and $\Pi {\cal D}_c $ is the identity  map in $\h_E^c(\M)$.
 \end{lemma}
\begin{proof} Let $n\geq 1$ and  $b \in \mathcal{FS}$. Using the  Kadison-Schwarz inequality for conditional expectations, we have
\begin{equation}\label{KS}
\E_{n-1}(|\E_n(b)-\E_{n-1}(b)|^2) \leq \E_{n-1}(|b|^2).
\end{equation}
We define $\Pi: \mathfrak{F} \to \h_E^c(\M)$  by setting:
\[
a=(a_n)_{n\geq 1}  \mapsto \sum_{n\geq 1} \E_n(a_n)- \E_{n-1}(a_n).
\]
It follows from \eqref{KS} that $s_c(\Pi(a)) \leq  \sigma_c(a)$ for every $a\in \mathfrak{F}$.
If $\mathfrak{F}$ is dense in $E^{\rm cond}(\M;\ell_2^c)$, then  $\Pi$ extends to a bounded linear map from $E^{\rm cond}(\M;\ell_2^c)$ onto $\h_E^c(\M)$ with $\|\Pi:  E^{\rm cond}(\M;\ell_2^c) \to \h_E^c(\M)\| \leq 1$. Clearly, if $x \in \h_E^c(\M)$,  we have $\Pi{\cal D}_c (x)=x$. This verifies the lemma for the case where $E$ is  separable. 

For the general case, 
we observe   that  from the fact that $\mathfrak{F}$ is dense in $(L_p +L_q)^{\rm cond}(\M;\ell_2^c)$,  the inequality on conditioned square functions above   can be restated as:
\begin{equation}\label{KS2}
\big| U{\cal D}_c(\Pi( y)) \big| \leq \big| U( y)\big|, \quad y\in (L_p + L_q)^{\rm cond}(\M;\ell_2^c).
\end{equation}

If $y \in E^{\rm cond}(\M;\ell_2^c)$, then  by definition  $y \in (L_p +L_q)^{\rm cond}(\M;\ell_2^c)$ and $U(y) \in E(\M \overline{\otimes} B(\ell_2(\mathbb{N}^2)))$.
By the separable case, $\Pi(y) \in \h_{L_p +L_q}^c(\M)$. Moreover, it follows from 
\eqref{KS2} that $U{\cal D}(\Pi(y)) \in E(\M \overline{\otimes} B(\ell_2(\mathbb{N}^2)))$.  This means, $\Pi(y)\in \h_E^c(\M)$ with $\|\Pi(y)\|_{\h_E^c}\leq \| y\|_{E^{\rm cond}(\M;\ell_2^c)}$.
\end{proof} 
 
 
  We refer to \cite{Cad-Ricard, Dirksen2, Jiao2, Jiao-Sukochev-Zanin-Zhou, JX, RW, Ran-Wu-Xu} for more information on noncommutative  Hardy spaces associated with symmetric spaces of measurable operators. In the sequel,  noncommutative column Hardy spaces associated with Orlicz  space $L_\Phi$ will be denoted by $\H_\Phi^c(\M)$ and $\h_\Phi^c(\M)$ while those associated with the Lorentz space $L_{p,q}$ will be denoted by $\H_{p,q}^c(\M)$ and $\h_{p,q}^c(\M)$.

We  conclude this subsection with a  description of  the dual space of the  Hardy scpace $\h_1^c(\M)$. A  martingale $x$ belongs to  the  \emph{column  little bmo space}   denoted by $\bmo^c(\M)$ if 
\[
\big\| x \big\|_{\bmo^c}= \max\Big\{ \big\|x_1\big\|_\infty  ; \sup_{m} \sup_{1\leq n\leq m }\big\| \E_n(|x_m-x_n|^2) \big\|_\infty^{1/2} \Big\}.
\] 
The Banach  space $(\bmo^c(\M), \|\cdot\|_{\bmo^c})$ was introduced  in \cite{Perrin} for finite case where it was shown that  it    coincides with the dual of the Hardy space 
$\h_1^c(\M)$.  The proof given there can be easily generalized to the semifinite case. We record this for further use:
\begin{equation} \label{h1-dual}
(\h_1^c(\M))^* =\bmo^c(\M)
\end{equation} 
with equivalent norms.
\subsection{Atomic decompositions for  martingale  Orlicz-Hardy spaces}

In this subsection, we will analyze  conditioned spaces and  conditioned Hardy spaces associated  with Orlicz spaces.  More specifically,  we will  describe  a type of atomic decomposition  in the context of Orlicz spaces. Results from this subsection  play key role in the next section. Toward this end, we will start from setting up some notations.

Throughout this subsection,  we always assume that $0<p \leq q <2$ and $\Phi$ is an Orlicz function that is $p$-convex and $q$-concave. First, we fix a positive Borel measure $\mu$ on the interval $[0,\infty)$ so that:
\begin{equation}\label{measure}
\Phi(t)\approx_{p,q}\int_0^\infty \min\{(ts)^p, (ts)^q\} \ d\mu(s).
\end{equation}
The existence   of such integral representation was proved  for convex functions  (see the proof of \cite[Lemma~6.2]{Jiao-Sukochev-Zanin}).  The argument given in \cite{Jiao-Sukochev-Zanin}   can be readily adjusted to include the more general  case of $p$-convex functions when $0<p<1$. 

 Next, we fix an Orlicz function $\Theta$ so that $L_\Phi$ admits the factorization:
\begin{equation}\label{Theta}
L_\Phi=L_2 \odot L_\Theta
\end{equation}
where the product $L_2 \odot L_\Theta$ is  the collection of all function $f$ that admit  factorization$f=gh$  with $g\in L_2$ and $h\in L_{\Theta}$.  The quasi-norm
on $L_2\odot L_\Theta$ is given by:
\[
\big\|f\big\|_{L_2 \odot L_\Theta} :=\inf\left\{\big\|g\big\|_{L_2}\big\|h \big\|_{L_\Theta}: g\in L_2,  h\in L_\Theta, f=gh\right\}.
\]

We note that the Orlicz function  $\Theta$ can be taken  to be  the inverse of  the function $t\mapsto t^{-1/2}\Phi^{-1}(t)$ for  $t>0$. We refer to \cite{Maligranda2} for more details.

We will also make use of the following function:
\begin{equation}\label{psi}
\Psi(t)= \int_0^\infty (t^{2-p}s^{-p} + t^{2-q}s^{-q})^{-1}\ d\mu(s)
\end{equation}
where $\mu$ is the positive Borel measure from  the representation of $\Phi$ in 
\eqref{measure}.  The reason for the consideration of $\Psi$  is summarized in the next lemma. The first three items are  straightforward generalizations of \cite[Proposition~3.3]{Ran-Wu-Zhou} while the last item can be deduced as in \cite[Lemma~3.4]{Ran-Wu-Zhou}.
\begin{lemma}\label{psi-prop}
\begin{enumerate}[{\rm(i)}]
\item $\Psi(t) \approx_{p,q} t^{-2} \Phi(t)$;
\item $\Theta\big( \Psi(t)^{-1/2}\big) \approx_{p,q} \Phi(t)$;
\item $t \mapsto \Psi(t^{1/2})$ is  operator monotone decreasing;
\item for any increasing sequence of positive operators $a_n \uparrow a$, we have:
\[
\sum_{n\geq 1} \T\big( (a_{n+1}^2-a_n^2) \Psi(a_{n+1})\big) \lesssim_{p,q} \T\big(\Phi(a)\big).
\]
\end{enumerate}
\end{lemma}

We now introduce  a concept of atoms for conditioned space constructed from  the Orlicz function space $L_\Phi$.
\begin{definition}  A  sequence $x \in L_\Phi(\M;\ell_2^c)$ is called \emph{an algebraic $L_\Phi^{c,\rm cond}$-atom} if it admits a factorization $x=\a\ .\ \b$ where
\begin{enumerate}[{\rm (i)}]
\item $\a= \sum_{j<n} \a_{n,j} \otimes e_{n,j}$ is a strictly lower triangular matrix in $L_2(\M \overline{\otimes} B(\ell_2(\mathbb N)))$ with 
\[
\big\|\a\big\|_2 =\big( \sum_{j<n} \|\a_{n,j}\|_2^2 \big)^{1/2} \leq 1;
\]
\item $\b \in L_{\Theta}^{\rm ad}(\M;\ell_2^c)$  with $\big\| \b \big\|_{L_\Theta(\M;\ell_2^c)}\leq 1$.
\end{enumerate}
\end{definition}

The above definition was motivated by   the case of $L_r^{\rm cond}(\M;\ell_2^c)$ for $1\leq r<2$ introduced for the first time  in \cite{Perrin2} and explored further in \cite{Hong-Junge-Parcet}. We also refer to \cite{Chen-Ran-Xu} where the same concept  was considered for the  case of conditioned Hardy space $\h_r^c(\M)$ for  the range $0<r\leq 1$.  More recently, the case of   Orlicz-conditioned  Hardy space $\h_\varphi^c(\M)$ associated with convex Orlicz function $\varphi$ was formulated in \cite{Ran-Wu-Zhou} in the context of $\varphi$-moments.

We  should  note   that since strictly lower triangular matrices were used  in the definition of algebraic atoms, it follows that if $x=(x_n)_{n\geq 1}$ is an algebraic $L_\Phi^{c,\rm cond}$-atom then $x_1=0$. 

The next  lemma can be   deduced   as  in  
 the first part of \cite[Theorem~3.6.10]{Perrin2} using the factorization  $L_\Phi=L_2 \odot L_\Theta$.  We include the argument   for  the sake  of completeness.
 \begin{lemma}\label{atom-inclusion}
 Every algebraic $L_\Phi^{c,\rm cond}$-atom  belongs to  $L_\Phi^{0,\rm cond}(\M;\ell_2^c)$.  More precisely,
if $x$ is an algebraic  $L_\Phi^{c,\rm cond}$-atom then
$\displaystyle{
\|x \|_{L_{\Phi}^{\rm cond}(\M;\ell_2^c)} \leq 1}$.
\end{lemma}
\begin{proof}
Let $x=\a\ .\ \b$ be an algebraic $L_\Phi^{c,\rm cond}$-atom. As observed earlier,   $x_1=0$ and  for  $n\geq 2$,
$x_n= \sum_{j<n} \a_{n,j} \b_j$.
Since $\b$ is adapted, it follows that
\[
\E_{n-1}|x_n|^2 =\sum_{m,j<n}  \b_m^* \E_{n-1}(\a_{n,m}^* \a_{n,j}) \b_j.
\]
From the property of the module map $u_{n-1}$, we have
\begin{align*}
\E_{n-1}|x_n|^2 \otimes e_{1,1}  &= \sum_{m,j <n}  (\b_m^* \otimes e_{1,1}) . u_{n-1}(\a_{n,m})^* u_{n-1}(\a_{n,j})  . (\b_j \otimes e_{1,1}) \\
&=| \sum_{j<n} u_{n-1}(\a_{n,j}).(\b_j \otimes e_{1,1})|^2.
\end{align*}
This implies that
\[
\sigma_c^2(x) \otimes e_{1,1} =\sum_{n\geq 2}
| \sum_{j<n} u_{n-1}(\a_{n,j}).(\b_j \otimes e_{1,1})|^2.
\]
We write further that 
\[
\sigma_c^2(x) \otimes e_{1,1} \otimes e_{1,1} =| \sum_{n\geq 2}
 \sum_{j<n} [u_{n-1}(\a_{n,j}).(\b_j \otimes e_{1,1})] \otimes e_{n,1} |^2.
\]
This allows us to deduce that
\[
\sigma_c(x) \otimes e_{1,1} \otimes e_{1,1} = \big|\big(u_{n-1}(\a_{n,j}) \big)_{j<n}\ . \ \sum_{k\geq 1} \b_k \otimes e_{1,1} \otimes  e_{k,1}  \big|
\]
where the strictly lower triangular  matrix $\widehat{\a}=(\big(u_{n-1}(\a_{n,j}) )_{j<n}$ takes its values  in $L_2(\M \overline{\otimes}  B(\ell_2(\mathbb{N})))$. We may view 
$\widehat{\a}$ as an operator affiliated with  $\mathfrak{M}=\M \overline{\otimes}  B(\ell_2(\mathbb{N}))) \overline{\otimes}  B(\ell_2(\mathbb{N})))$. If $\mu(\cdot)$  denote  the generalized  singular  number relative  to $\mathfrak{M}$ equipped with its natural trace, then
\[
\big\| \sigma_c(x)\big\|_{L_\Phi(\M)} = \|\mu(\sigma_c(x) \otimes e_{1,1} \otimes e_{1,1}) \big\|_{L_\Phi}.
\]
It follows from \cite[Theorem~4.2]{FK} that
\begin{align*}
\big\| \sigma_c(x)\big\|_{L_\Phi(\M)} 
&\leq  \|\mu(\widehat{\a}) \big\|_{2} .\big\| \mu\big( \sum_{k\geq 1} \b_k \otimes e_{1,1} \otimes e_{k,1}\big) \big\|_{L_\Theta}\\
&= \big\| \widehat{\a}\big\|_{2} . \big\| \big(\sum_{k\geq 1} |\b_k|^2\big)^{1/2} \big\|_{L_\Theta}\\
&\leq  \big(\sum_{j<n} \| \a_{n,j}\|_2^2 \big)^{1/2}  . \big\| \big(\sum_{k\geq 1} |\b_k|^2\big)^{1/2} \big\|_{L_\Theta}.
\end{align*} 
This proves that $\| x \|_{L_\Phi^{\rm cond}(\M;\ell_2^c)}\leq 1$.
\end{proof}
Using the above notion of atoms, we may naturally  consider the following concept  of atomic decompositions:
\begin{definition}
 A sequence  $x \in L_\Phi(\M;\ell_2^c)$ is said to admit an \emph{algebraic $L_\Phi^{c,\rm cond}$-atomic decomposition} if 
\[
x =\sum_{k} \lambda_k a^{(k)},
\] 
where for each $k$, $a^{(k)}$ is either an algebraic $L_\Phi^{c,\rm cond}$-atom   or $a^{(k)}$  belongs to the unit ball of  the conditioned space $L_\Phi^c(\M;\E_1)$ and $\lambda_k \in \mathbb{C}$ satisfying $\sum_k|\lambda_k|^p <\infty$ for $0<p<1$ and $\sum_k|\lambda_k| <\infty$ for $1\leq p<2$. Since $\Phi$ is $p$-convex and algebraic $L_\Phi^{c,\rm cond}$-atoms belong  to the unit ball of $L_\Phi^{\rm cond}(\M;\ell_2^c)$, it follows that  if $x$ admits an algebraic $L_\Phi^{c,\rm cond}$-atomic decomposition then it belongs to $L_\Phi^{\rm cond}(\M;\ell_2^c)$.
\end{definition}

Following \cite{Perrin2},  the corresponding \emph{ algebraic atomic column conditioned space} $L_{\Phi,\alg}^{ \rm cond}(\M;\ell_2^c)$ is defined to be  the completion of  the space of all $x$ that  admit  algebraic $L_\Phi^{c,\rm cond}$-atomic decompositions  in the space $L_\Phi^{\rm cond}(\M;\ell_2^c)$. If $x$ admits an  algebraic $L_\Phi^{c,\rm cond}$-atomic decomposition, we set:
\[
\|x\|_{L_{\Phi,\alg}^{\rm cond}(\M;\ell_2^c)}=\inf\Big(\sum_k |\lambda_k|^p\Big)^{1/p}
\quad \text{for} \ 0<p\leq 1\]
and
\[
\|x\|_{L_{\Phi,\alg}^{\rm cond}(\M;\ell_2^c)}=\inf\sum_k |\lambda_k|
\quad \text{for} \ 1<p<2,\]
where the infimum is taken over all decompositions of $x$ as described above. We refer  the reader to \cite{Chen-Ran-Xu} for a  more  in depth  discussion on  the need to separate the two cases.

The following result generalizes  the atomic decomposition of conditioned $L_p$-spaces from \cite{Perrin2}  in two directions: it is valid for  Orlicz spaces and also cover the quasi-Banach space range.  This will play a crucial role in the next section.  The approach of \cite{Perrin2} was by duality which is not  applicable to the present situation since we are dealing with not necessarily convex functions. Our constructive proof given  below combined ideas  from  \cite{Chen-Ran-Xu}  and  the case of Hardy spaces associated with convex Orlicz functions considered in \cite{Ran-Wu-Zhou}. This may be of independent interest.
\begin{theorem}\label{main-algebraic}
Let $\Phi$ be  $p$-convex and $q$-concave for $0<p\leq q<2$. Then the two spaces
$
L_{\Phi,\alg}^{ \rm cond}(\M;\ell_2^c)$ and  $ L_\Phi^{\rm cond}(\M;\ell_2^c)$ coincide (with constant of isomorphism depending only on $p$ and $q$).

More precisely,  every  $ x \in \mathfrak{F}^0$ admits a factorization $x= \a \, .\, \b$  where  $\a$   is a strictly lower triangular matrix  in $L_2(\M \overline{\otimes} B(\ell_2(\mathbb{N})))$ and $\b  \in L_\Theta^{\rm ad}(\M,\ell_2^c)$ satisfying:
\[
\big\| \a\big\|_2 \, .\, \big\| \b\big\|_{L_\Theta(\M;\ell_2^c)} \lesssim_{p,q} \big\| x \big\|_{L_\Phi^{\rm cond}(\M;\ell_2^c)}.
\]
\end{theorem}
\begin{proof}
First, we recall from earlier discussion   that $L_{\Phi,\alg}^{ \rm cond}(\M;\ell_2^c)  \subseteq L_\Phi^{\rm cond}(\M;\ell_2^c)$. Thus,  we only need to verify one inclusion.
This will  be deduced from  the second part of the theorem.

 Let $ x=(x_n)_{n=1}^N \in \mathfrak{F}^0$.   The construction below is an adaptation of the argument used in \cite{Ran-Wu-Zhou}. By definition,  $x_1=0$ and there exists a projection $e\in \M_1$ with $\T(e)<\infty$  such that  for $2\leq n\leq N$, $x_n=ex_n e$.
 
  First, we note that for $j\geq 2$,
we have $\sigma_{c,j}(x) \in e\M_{j-1}e$. By approximation, we  may assume that each  of the $\sigma_{c,j}(x)$'s  is invertible with bounded inverse in $e\M e$.  Below, we simply write $\sigma_j$ for $\sigma_{c,j}(x)$ and the function $\Psi$ is as defined in \eqref{psi}.  Let $\l>0$ to be determined later. For $n\geq 2$, we write
\begin{align*}
x_n  &=x_n \Psi(\l\sigma_n)\Psi(\l\sigma_n)^{-1}\\
&=x_n\Psi(\l\sigma_n)\big[ \Psi(\l\sigma_2)^{-1} + \sum_{3\leq m\leq n} \Psi(\l\sigma_m)^{-1}-\Psi(\l\sigma_{m-1})^{-1}\big]\\
&=x_n \Psi(\l\sigma_n)\Psi(\l\sigma_2)^{-1} + \sum_{2\leq j<n} x_n \Psi(\l\sigma_n)\big(\Psi(\l\sigma_{j+1})^{-1} -\Psi(\l\sigma_j)^{-1}\big).
\end{align*}
We define the  strictly lower triangular matrix $\a$ by  setting:
\begin{equation}\label{spliting-1}
\begin{cases}
 \alpha_{n,1} &:=x_n\Psi(\l\sigma_n)\Psi(\l\sigma_2)^{-1/2};\\
 \alpha_{n,j}&:=\displaystyle{x_n\Psi(\l\sigma_n)\left(\Psi(\l\sigma_{j+1})^{-1}-\Psi(\l\sigma_{j})^{-1}\right)^{1/2},\quad\text{for $2 \leq j<n$}.}
 \end{cases}
 \end{equation}
 We should point out here that since the function $t\mapsto \Psi(t^{1/2})$   is operator monotone decreasing and $\lambda^2\sigma_{j}^2 \leq \lambda^2 \sigma_{j+1}^2$, we have $\Psi(\lambda\sigma_{j+1}) \leq  \Psi(\lambda\sigma_{j})$. Taking inverses, 
 $\Psi(\l\sigma_{j+1})^{-1}-\Psi(\l\sigma_{j})^{-1}\geq 0$. Thus,  taking $1/2$-power in the expression above is  justified.

 The column sequence is defined by setting:
 \begin{equation}\label{spliting-2}
 \begin{cases}
 \beta_{1,1} &:=\displaystyle{\Psi(\l\sigma_2)^{-1/2};}\\
 \beta_{m,1}&:=\displaystyle{\left(\Psi(\l\sigma_{m+1})^{-1}-\Psi(\l\sigma_{m})^{-1}\right)^{1/2}, \quad \text{for  $m\geq 2$}}.
 \end{cases}
 \end{equation}
 Then $\alpha=(\alpha_{n,j})_{j<n}$ is a strictly lower triangular matrix and   $\b=(\b_{m,1})_{m\geq 1}$ is an adapted  sequence. Moreover,   for every $n\geq 2$, it clearly follows from  the definition that
 \[
 (\alpha.\beta)_{n,1}= \sum_{j\geq 1} \alpha_{n,j}. \beta_{j,1}=x_n.
 \]
 That is,  we have the factorization $x= \a\, .\,  \b$.
 We claim that  the product $\a\, . \, \b$  satisfies the desired norm estimate.  We begin with the $L_2$-norm of $\a$. 
 \begin{align*}
 \big\|\a\big\|_2^2 &=\sum_{n\geq 2} \sum_{1\leq j<n} \big\|\a_{n,j}\big\|_2^2\\
 &=\sum_{n\geq 2} \T\big(\Psi(\l\sigma_n) |x_n|^2 \Psi(\l\sigma_n) \big[ \Psi(\l\sigma_2)^{-1} +\sum_{2\leq j<n} \Psi(\l\sigma_{j+1})^{-1}-\Psi(\l\sigma_j)^{-1}\big]\big)\\
 &=\sum_{n\geq 2}\T\big(|x_n|^2\Psi(\l\sigma_n)\big).
 \end{align*}
 Since $(\Psi(\l \sigma_n))$ is a predictable sequence, we have 
 \begin{align*}
 \big\|\a\big\|_2^2&=\sum_{n\geq 2} \T\big(\E_{n-1}(|x_n|^2)  \Psi(\l\sigma_n) \big)\\
 &=\sum_{n\geq 2} \T\big( (\sigma_n^2- \sigma_{n-1}^2)\Psi(\l\sigma_n)\big)\\
  &=\l^{-2}\sum_{n\geq 2} \T\big( ((\l\sigma_n)^2-(\l \sigma_{n-1})^2)\Psi(\l\sigma_n)\big).
 \end{align*}
 We may deduce from Lemma~\ref{psi-prop}(iv) that there is a constant $C_{p,q}$ so that 
 \begin{equation}\label{alpha}
 \big\|\a\big\|_2^2 \leq C_{p,q} \l^{-2} \T\big(\Phi(\l\sigma_c(x))\big). 
\end{equation}
 On the other hand, Lemma~\ref{psi-prop}(ii) implies that there exists a constant $C_{p,q}'$ so that
 \begin{equation}\label{beta}
 \T\Big[ \Theta\Big( \big(\sum_{m\geq 1} |\beta_{m,1}|^2\big)^{1/2} \Big)\Big] =\T \Big[
 \Theta\Big( \Psi (\l\sigma_c(x))^{-1/2} \Big)\Big] \leq C_{p,q}'\T\big(\Phi(\l\sigma_c(x))\big).
 \end{equation}
 Let $K_{p,q}=\max\{C_{p,q}, C_{p,q}'\} +1$.
 Since $t\mapsto t^{-p}\Phi(t)$ is non-decreasing, one can easily verify that 
  $\Phi(t)\leq K_{p,q}^{-1} \Phi(K_{p,q}^{1/p}t)$. Using this fact, we get from \eqref{alpha} and \eqref{beta} that
 \begin{equation*}
 \big\|\a\big\|_2^2 \leq  \l^{-2} \T\big(\Phi( K_{p,q}^{1/p}\l\sigma_c(x))\big)
\end{equation*} 
and
\begin{equation*}
 \T\Big[ \Theta\Big( \big(\sum_{m\geq 1} |\beta_{m,1}|^2\big)^{1/2} \Big)\Big] \leq \T\big(\Phi( K_{p,q}^{1/p}\l\sigma_c(x))\big).
\end{equation*}
Choose $\lambda$ so that $\T\big(\Phi( K_{p,q}^{1/p}\l\sigma_c(x))\big)\leq 1$. This can be achieved with $\l^{-1}= K_{p,q}^{1/p} \big\|\sigma_c(x) \big\|_{L_\Phi(\M)}$.
With the above choice of $\lambda$, we clearly have 
$
\big\| \a\big\|_2 \leq  K_{p,q}^{1/p} \big\| x \big\|_{L_\Phi^{\rm cond}(\M;\ell_2^c)}$ and $\big\|\b\big\|_{L_\Theta(\M;\ell_2^c)} \leq 1$.
This proves the desired estimate and therefore the second part of the theorem.

To conclude the proof,  we apply  the case of $\mathfrak{F}^0$ with  the direct sum \eqref{direct-sum}. We see that every $y \in \mathfrak{F}$ admits an algebraic $L_\Phi^{c,\rm cond}$-atomic decomposition  with $\|y\|_{L_{\Phi,\alg}^{\rm cond}(\M;\ell_2^c)} \lesssim_{p,q} \|y\|_{L_\Phi^{\rm cond}(\M;\ell_2^c)}$. Since $\mathfrak{F}$ is dense in $L_\Phi^{\rm cond}(\M;\ell_2^c)$, it follows that
$L_{\Phi,\alg}^{\rm cond}(\M;\ell_2^c)=L_\Phi^{\rm cond}(\M;\ell_2^c)$.
 \end{proof}
 \begin{remark}\label{moment-algebraic}
 Using $\lambda=1$ in the proof above, we also obtain a moment version of the preceding theorem: given a sequence $x=(x_n)_{n\geq 1} \in \mathfrak{F}^0$,  the column matrix $\overline{x}=\sum_{n\geq 1} x_n \otimes e_{n,1}$ admits a factorization $\overline{x}=\alpha\ .\ \beta$  with $\alpha$ is a strictly lower triangular in $L_2(\M \overline{\otimes} B(\ell_2(\mathbb N)))$ and $\beta$ is an adapted  column matrix satisfying
 \[
 \big\|\alpha\big\|_2 + \T\big[\Theta\big( (\sum_{n\geq 1} |\beta_{n,1}|^2 )^{1/2}\big)\big] \lesssim_{p,q} \T\big[\Phi\big( \sigma_c(x)\big)\big]. 
 \]
 
 \end{remark}

We consider now the version of algebraic atomic decomposition for martingale Hardy  spaces. The next consideration  generalizes  notions from \cite{Chen-Ran-Xu}  and \cite{Ran-Wu-Zhou} for  non necessarily convex Orlicz functions.
 
 \begin{definition}\label{alg}
 Let $0<p< q<2$ and $\Phi$ be  an Orlicz function that is $p$-convex and $q$-concave.
 An operator $x \in L_\Phi(\M)$ is called an {\em algebraic $\h_\Phi^c$-atom}, whenever it can be written in the form $x=\sum_{n\geq 1} y_nb_n$, with $y_n$ and $b_n$ satisfying the following conditions:
\begin{enumerate}[{\rm (i)}]

\item  $\mathcal{E}_n (y_n) =0$ and $b_n \in L_\Theta(\M_n)$ for all $n\geq 1$;

\item  $\displaystyle{\sum_{n\geq 1} \big\|y_n\big\|_2^2 \leq 1}$ and $\displaystyle{\Big\| \Big( \sum_{n\geq 1} |b_n|^2 \Big)^{1/2}\Big\|_{L_\Theta(\M)} \leq 1}$.
\end{enumerate}
\end{definition}

Following \cite{Chen-Ran-Xu, Ran-Wu-Zhou}, this concept of atoms naturally  leads to the consideration  of the corresponding  atomic decomposition for   conditioned Orlicz-Hardy spaces:
 we say that an operator $x \in L_\Phi(\M)$ admits an algebraic $\h_\Phi^c$-atomic decomposition if 
\[
x =\sum_{k} \lambda_k a_k,
\] 
where for each $k$, $a_k$ is an algebraic $\h_\Phi^c$-atom  or an element of the unit ball of $L_\Phi(\M_1)$, and $\lambda_k \in \mathbb{C}$ satisfying $\sum_k|\lambda_k|^p <\infty$ for $0<p \leq 1$ and 
 $\sum_k|\lambda_k|<\infty$ for $1<p<2$. The corresponding \emph{ algebraic atomic column martingale Hardy space} $\h_{\Phi,\alg}^c(\M)$ is defined to be  the space of all $x$ which admit a algebraic $\h_\Phi^c$-atomic decomposition and is equipped with
\[
\|x\|_{\h_{\Phi,\alg}^c}=\inf\Big(\sum_k |\lambda_k|^p\Big)^{1/p} \quad \text{for $0<p\le1$}
\]
and
\[
\|x\|_{\h_{\Phi,\alg}^c}=\inf\,\sum_k |\lambda_k| \quad \text{for $1< p<2$},
\]
where the infimum are   taken over all  decompositions of $x$ as described above.
 
 We note that the above concepts were introduced in \cite{Ran-Wu-Zhou} for  the case where $\Phi$ is a convex function. Our main focus here is the case where $\Phi$ is $p$-convex with $0<p<1$.

 The next result is an extension of \cite[Theorem~3.10]{Chen-Ran-Xu} to the case of Orlicz function spaces. It follows immediately from Theorem~\ref{main-algebraic} and   the complementation  result stated in Lemma~\ref{complemented}.
 
 \begin{corollary}\label{decomp-marting}
Let $0<p<q<2$ and $\Phi$ is an Orlicz function that is $p$-convex and $q$-concave. Then
 \[
 \h_\Phi^{c}(\M)= \h_{\Phi,\alg}^c(\M).
 \]
 More precisely, if $ x \in \h_\Phi^c(\M)$, then $x$ admits a  unique decomposition
 $x=x_1 +y$ where $x_1 \in L_\Phi(\M_1)$ and $y$ is a scalar multiple of an algebraic $\h_{\Phi}^c$-atom. 
 \end{corollary}
 Indeed, if $\Pi: L_\Phi^{\rm cond}(\M;\ell_2^c) \to \h_\Phi^c(\M)$ denotes the  norm  one  projection described in  Lemma~\ref{complemented} and $x=\a\, . \, \beta$  is a algebraic $L_\Phi^{c,{ \rm cond}}$-atom with $\a$ being a strictly lower triangular matrix in  $L_2(\M \overline{\otimes} B(\ell_2(\mathbb N)))$ and $\b$ is an adapted column matrix in $L_{\Theta}(\M;\ell_2^c)$, then  we have 
 \begin{align*}
 \Pi(x) &= \sum_n \sum_{n>k} d_n(\a_{n,k}) \beta_{k,1}\\
 &=\sum_k \big(\sum_{n>k} d_n(\a_{n,k}) \big) \beta_{k,1}\\
 &=\sum_k  a_k \beta_{k,1}.
 \end{align*}
 Clearly, we have for every $
 k\geq 1$, $\E_k(a_k)=0$. Moreover,  $\sum_{k\geq 1} \|a_k\|_2^2 \leq  \|\a\|_2^2$. This shows  in particular that $\Pi(x)$ is an algebraic $\h_\Phi^c$-atom.   The assertions in the corollary follow from combining this fact with Theorem~\ref{main-algebraic}  and direct sum. \qed
 
 \medskip
 
 We conclude this section  with   a   companion of Theorem~\ref{main-algebraic} for spaces of adapted sequences. It may be  viewed  as  the algebraic atomic decompositions  for  spaces of adapted sequences. This will be used in the next section.
  
 \begin{proposition}\label{factorization-adapted} Let $\Phi$ be a $p$-convex and $q$-concave Orlicz function for $0<p<q<2$. If $x=(x_n)_{n\geq 1}$ is a sequence in $L_\Phi^{\rm ad}(\M;\ell_2^c)$
 then there exists a sequence $\b=(\beta_n)_{n\geq 1}$ in $L_\Theta^{\rm ad}(\M;\ell_2^c)$ and a  lower triangular matrix $\a$ with the following properties:
 \begin{enumerate}[{\rm (i)}]
\item $\a \in L_2(\M\overline{\otimes} B(\ell_2(\mathbb N)))$;
\item for every $1\leq j\leq n$,  $\a_{n,j} \in L_2(\M_n)$;
\item $x=\a\,  . \, \b$;
\item $\big\|\a \big\|_2 \, .  \, \big \|\b \big\|_{L_\Theta(\M;\ell_2^c)} \lesssim_{p,q} \big\| x\big\|_{L_\Phi(\M;\ell_2^c)}$.
 \end{enumerate}
 Conversely, any sequence $x$ admitting  a factorization  as above belongs to $L_\Phi^{\rm ad}(\M;\ell_2^c)$ with
 \[
 \big\| x\big\|_{L_\Phi(\M;\ell_2^c)} \leq \big\|\a \big\|_2 \, .  \, \big \|\b \big\|_{L_\Theta(\M;\ell_2^c)}.
 \]
 \end{proposition}
 
 First, we note that since square functions are well-defined for  elements of $L_\Phi^{\rm ad}(\M; \ell_2^c)$, reduction to finite sequences  or sequences of finite supports is not necessary. 
 \begin{proof}[Sketch of the proof] Let $x=(x_n)_{n\geq 1}\in L_\Phi^{\rm ad}(\M;\ell_2^c)$. The construction is an adaptation of   the proof of Theorem~\ref{main-algebraic} but using square functions in place of conditioned square functions.
 
  For each $n\geq 1$,  $\cal{S}_{c,n}(x) \in L_{\Phi}(\M_n)$.  That is,  $(\cal{S}_{c,n}(x))_{n\geq 1}$ is an adapted sequence.  We simply write $\varsigma_n$ for $\cal{S}_{c,n}(x)$. As before, we  may assume that  the $\varsigma_n$'s are invertible with bounded inverse. 
  Similarly, $(\Psi(\varsigma_n))_{n\geq 1}$ is an adapted sequence with the $\Psi(\varsigma_n)$'s being invertible.
  
   As in the proof of Theorem~\ref{main-algebraic}, fix $\l>0$ and  set
\begin{equation*}
\begin{cases}
 \alpha_{n,1} &:=x_n\Psi(\l\varsigma_n)\Psi(\l\varsigma_1)^{-1/2};\\
 \alpha_{n,j}&:=\displaystyle{x_n\Psi(\l\varsigma_n)\left(\Psi(\l\varsigma_{j})^{-1}-\Psi(\l\varsigma_{j-1})^{-1}\right)^{1/2},\quad\text{for $2 \leq j \leq n$};}
 \end{cases}
 \end{equation*}
 and
 \begin{equation*}
 \begin{cases}
 \beta_{1,1} &:=\displaystyle{\Psi(\l\varsigma_1)^{-1/2};}\\
 \beta_{m,1}&:=\displaystyle{\left(\Psi(\l\varsigma_{m})^{-1}-\Psi(\l\varsigma_{m-1})^{-1}\right)^{1/2}, \quad \text{for  $m\geq 2$}}.
 \end{cases}
\end{equation*}
A slight difference here is that unlike in the proof of Theorem~\ref{main-algebraic},  we do not make the indexing shift. The factorization  $x=\a\, .\, \b$ is straightforward.  Clearly, $(\beta_{m,1})_{m\geq 1}$ is an adapted sequence. Moreover, as $x_n \in L_\Phi(\M_n)$, it  follows  that $\a_{n,j}$ is  affiliated with  $\M_n$ and  $(\a_{n,j})_{1\leq j\leq n}$ is a lower triangular matrix.

We may choose $\l$ exactly as in the proof of Theorem~\ref{main-algebraic}. That is, 
$\lambda^{-1}=K_{p,q}^{1/p} \| \cal{S}_c(x) \|_{L_\Phi(\M)}$. With this  choice, 
  the  verification of the fact that $\big\|\a\big\|_2\, . \,   \big\|\b\big\|_{L_\Theta(\M;\ell_2^c)} \lesssim_{p,q} \big\|x \big\|_{L_\Phi(\M;\ell_2^c)}$ follows the same reasoning  as in the proof of Theorem~\ref{main-algebraic} and is left to the reader.
  
  On the other hand, if $x$ admits a factorization $x=\a\, .\, \b$, then from the fact that $L_\Phi =L_2 \odot  L_\Theta$, we may conclude as in the proof of Lemma~\ref{atom-inclusion} that 
  \[
  \|x \|_{L_\Phi(\M;\ell_2^c)}  \leq \|\a\|_2\,  .  \,  \|\b\|_{L_\Theta(\M;\ell_2^c)}.
  \]
   The proof is complete.
 \end{proof}
 
 \begin{remark} Using factorizations of operator-valued  triangular matrices (or more generally, elements of Hardy spaces associated with semifinite version of subdiagonal algebras in the sense of Arveson), the existence of a factorization $x=\a\, . \, \b$,  where $\a$ is a lower triangular and $\b$  is a column matrix, is clear. We refer to \cite{PX} for this fact. The main point of Proposition~\ref{factorization-adapted} is that when $x$ is an adapted sequence, we can choose the sequence $\b$ to be adapted and the matrix $\a$ to satisfy  
 the extra property stated in item~(ii). These additional  facts  are very  important  in the next section.
 \end{remark}
%

\section{Interpolation of conditioned spaces}
This section is devoted to  interpolation spaces  between  noncommutative column/row conditioned Hardy spaces and related spaces.
Our main references for interpolation theory are the books  \cite{BENSHA,BL, KPS}.

 Since we will be  concerned with $\h_p^c(\M)$ when $0<p<1$, we will consider  the more general framework of quasi-Banach spaces. We begin with some basic definitions.

  Let $(A_0, A_1)$ be a  compatible couple of quasi-Banach spaces  in the sense that both $A_0$ and $A_1$ embed continuously into some topological vector space $\mathcal{Z}$. This allows us to  define the spaces $A_0 \cap A_1$  and $A_0 +A_1$. These are quasi-Banach spaces when equipped with  quasi-norms:
\[
\big\|x  \big\|_{A_0 \cap A_1}=\max\Big\{ \big\|x  \big\|_{A_0 } , \big\|x  \big\|_{ A_1}\Big\}
\]
and
\[
\big\|x  \big\|_{A_0 + A_1}=\inf\Big\{ \big\|x_0  \big\|_{A_0 } + \big\|x_1  \big\|_{ A_1}: \, x=x_0 +x_1,\,  x_0 \in A_0,\,  x_1 \in A_1\Big\},
\]
respectively.  
\begin{definition}\label{def-interpolation}
A  quasi-Banach space $A$ is called  an \emph{interpolation space} for the couple $(A_0, A_1)$  if $A_0 \cap A_1 \subseteq A \subseteq A_0 +A_1$ and whenever  a bounded linear operator $T: A_0 +A_1\to A_0 +A_1$ is such that $T(A_0) \subseteq A_0$ and $T(A_1) \subseteq A_1$, we have $T(A)\subseteq A$
and 
\[
\big\|T:A\to A \big\|\leq c\max\left\{\big\|T :A_0 \to A_0\big\| , \big\|T: A_1 \to A_1\big\| \right\}
\] 
for some constant $c$. 
\end{definition}
If $A$ is an interpolation space for the couple $(A_0, A_1)$, we write $A\in {\rm Int}(A_0,A_1)$. Below, we are mostly interested in  two  well-known specific interpolation methods generally referred to as  real  method and complex method.
  

We begin with   a short  discussion   of the  \emph{real interpolation} method.
 A fundamental notion  for the construction of   real interpolation spaces is the  \emph{$K$-functional} which we now describe. For $x \in A_0 +A_1$, we define the $K$-functional  by setting for $t>0$,
\[
K(x, t) =K\big(x,t; A_0,A_1\big)=\inf\Big\{ \big\|x_0  \big\|_{A_0 } + t\big\|x_1  \big\|_{ A_1}:\,  x=x_0 +x_1,\,  x_0 \in A_0,\,  x_1 \in A_1\Big\}.
\]
Note that  each $t>0$, $x \mapsto K(x,t)$ gives an equivalent  quasi-norm on $A_0  +A_1$. There is also the dual notion of $J$-functionals which is defined for $y \in A_0 \cap A_1$ and $t>0$,
\[
J(y, t) =J\big(y,t; A_0,A_1\big)=\max\Big\{ \big\| y  \big\|_{A_0 } , t\big\|y \big\|_{ A_1}\Big\}.
\]

If $0<\theta<1$ and $0< \g<\infty$, we recall the real interpolation space $A_{\theta, \g}=(A_0, A_1)_{\theta, \g}$ by $x \in A_{\theta,\g}$ if and only if

\[
\big\| x \big\|_{(A_0,A_1)_{\theta, \g}} =\Big( \int_0^\infty \big(t^{-\theta}K\big(x, t; A_0,A_1\big)\big)^{\g }\ \frac{dt}{t} \Big)^{1/\g} <\infty.
\]
If $\g=\infty$, we define $ x \in A_{\theta,\infty}$ if and only if 
\[
\big\| x\big\|_{(A_0, A_1)_{\theta, \infty}}= \sup_{t>0} t^{-\theta} K(x, t; A_0, A_1)<\infty.
\]
For $0<\theta<1$ and $0<\g\leq \infty$, $\|\cdot\|_{\theta,\g}$ is a  quasi-norm and $(A_{\theta,\g}, \|\cdot\|_{\theta,\g})$ is a quasi-Banach space. Moreover,
the space $A_{\theta,\g}$ is an interpolation space for the couple $(A_0, A_1)$  in the sense of Definition~\ref{def-interpolation}. There is also an equivalent description of $A_{\theta,\gamma}$  using the $J$-functionals for which we refer to \cite{BL, KPS} for the exact formulation. 

\smallskip

It is worth noting  that the  real interpolation method is well understood for the couple $(L_{p_0}, L_{p_1})$ for both   the classical case and the  noncommutative case. We record here that
Lorentz spaces can be realized as  real interpolation spaces  for  the couple $(L_{p_0}, L_{p_1})$. More precisely, 
 if $\cal N$  is a semifinite von Neumann algebra, $0<p_0<p_1\leq \infty$, $0<\theta<1$, and $0<q\leq \infty$ then,   up to equivalent  quasi-norms (independent of $\cal N$),  
\[
\big(L_{p_0}(\cal N), L_{p_1}(\cal N)\big)_{\theta,q}= L_{p,q}(\cal N)
\]
where $1/p=(1-\theta)p_0 +\theta/p_1$. In particular,  we have 
\[\big(L_{p_0}(\cal N), L_{p_1}(\cal N)\big)_{\theta,p}= L_{p}(\cal N)
\] with equivalent  quasi-norms. These facts can be found in \cite{PX}.

\medskip

Wolff's interpolation theorem will be used repeatedly throughout  the next subsection. We record it here for convenience.
\begin{theorem}[{\cite[Theorem~1]{Wolff}}] Let  $B_i$ ($i=1,2,3,4$) be quasi-Banach spaces such that $B_1 \cap B_4$ is dense in $B_j$ ($j=2,3$) and  satisfy:
\[
B_2=(B_1, B_3)_{\phi, r} \ \ \text{and}\ \ B_3=(B_2, B_4)_{\theta, q}
\]
for $0<\phi,\theta<1$ and  $0<r,q\leq \infty$. Then 
\[
B_2=(B_1, B_4)_{\xi,r} \ \ \text{and}\ \  B_3=(B_1, B_4)_{\zeta, q}
\]
where $\displaystyle{\xi=\frac{\phi\theta}{1-\phi +\phi\theta}}$ and $\displaystyle{\zeta=\frac{\theta}{1-\phi+\phi\theta}}$.
\end{theorem}

In order to make the presentation below more concise, we introduce the following terminology.
\begin{definition} 
A family of  quasi-Banach spaces $(A_{p,\g})_{p,\g \in (0,\infty]}$ is said to form a \emph{real interpolation scale} on an interval $I \subseteq \mathbb{R}\cup \{\infty\}$ if for every $p,q \in I$, $0<\g_1,\g_2,\g \leq \infty$, $0<\theta<1$, and $1/r=(1-\theta)p + \theta/q$,
\[
A_{r, \g}= (A_{p,\g_1}, A_{q,\g_2})_{\theta,\g}
\]
with equivalent quasi-norms.
\end{definition}

The next result may be viewed as a version of Wolff's interpolation theorem at the level of family of real interpolation scale.

\begin{lemma}\label{union}
 Assume that a family of quasi-Banach spaces $(A_{p,\g})_{p,\g \in (0,\infty]}$  forms a \emph{real interpolation scale} on  two  different intervals $I $ and $J$.  If  $|J \cap I |>1$,  then 
$(A_{p,\g})_{p,\g \in (0,\infty]}$  forms a real interpolation scale on $I\cup J$.
\end{lemma}

\begin{proof}  We may assume that $I$ and $J$ are closed intervals. As $|I\cap J|>1$,  we may assume that $\sup I >\inf J $ and $I\cap J=[w_1, w_2]$ where $w_1=\inf J$ and $w_2=\sup  I$. Fix $p \in I \setminus J$, $q \in J\setminus  I$,  and $0<\g_1,\g_2,\g\leq \infty$. We divide the proof into three cases.

\smallskip

$\diamond$ Case~1.   Assume that  $r_1 \in (w_1, w_2)$.

Since both $p$ and $w_2$ belong to $I$, by assumption, for any given $0<\g_3\leq \infty$,
\[
A_{r_1,\g}=( A_{p,\g_1}, A_{w_2,\g_3})_{\theta_1, \g}
\] 
 where $1/{r_1}=(1-\theta_1)/p +\theta_1/{w_2}$.

On the other hand, as $r_1$ and $q$ belong  to $J$ and $r_1<w_2<q$,  the assumption also gives that 
\[
A_{w_2,\g_3} =(A_{r_1,\g}, A_{q,\g_2})_{\psi_1, \g_3}
\]
where $1/{w_2}= (1-\psi_1)/r_1 + \psi_1/q$.
Applying Wolff's interpolation theorem, with $B_1= A_{p,\g_1}$, $B_2=A_{r_1,\g}$,
$B_3=A_{w_2,\g_3}$, and $B_4=A_{q,\g_2}$, we deduce that 
\[
A_{r_1, \g}=  (A_{p,\g_1}, A_{q,\g_2})_{\theta, \g}
\]
where $\displaystyle{\theta=\frac{\theta_1\psi_1}{1-\theta_1 +\theta_1\psi_1}}$. One can readily  verify that $1/r_1=(1-\theta)/p +\theta/q$.

\medskip

$\diamond$ Case~2.  Assume that $r_2 \in (p,  w_1]$. 

 Fix $w_3 \in (w_1, w_2)$. Since $p<r_2<w_3$ and $p , w_3 \in I$, by assumption,
 we have for every $0<\g_3\leq \infty$,
 \[
A_{r_2,\g}=( A_{p,\g_1}, A_{w_3,\g_3})_{\theta_2, \g}
\] 
where $1/{r_2}=(1-\theta_2)/p + \theta_2/{w_3}$.  Using the previous case with $r_2$ in place of $p$ and $w_3$ in place of $r_1$, we have 
\[
A_{w_3,\g_3} =( A_{r_2,\g}, A_{q,\g_2})_{\psi_2, \g_3}
\]
with $1/w_3=(1-\psi_2)/{r_2} +  \psi_2/q$. 
 Using  Wolff's interpolation theorem with $B_1=A_{p,\g_1}$, $B_2=A_{r_2,\g}$,
 $B_3=A_{w_3,\g_3}$, and $B_4= A_{q,\g_2}$, we obtain that
 \[
 A_{r_2,\g}=(A_{p,\g_1}, A_{q,\g_2})_{\theta,\g}
 \]
where  $\displaystyle{\theta=\frac{\theta_2\psi_2}{1-\theta_2 +\theta_2\psi_2}}$. As before, one can verify that $1/{r_2}= (1-\theta)/p + \theta/q$.
 
 \medskip
 
 $\diamond$ Case~3. Assume $r_3 \in [w_2, q)$. 
 
  As in the previous case, let  $w_3  \in (w_1, w_2)$. Since $w_3, r_3, q \in I$, we have 
\[
A_{r_3,\g}=  (A_{w_3, \g_3},  A_{q,\g_2})_{\theta_3, \g}
\]
where $1/{r_3}= (1-\theta_3)/{w_3} + \theta_3/q$.
 Next, we apply  Case~1 with $w_3$ in  place of $r_1$ in order to get that
\[
A_{w_3, \g_3}=  ( A_{p,\g_1}, A_{q,\g_2})_{\psi_3, \g_3}
\]
with $1/{w_3}=(1-\theta_3)/p +\theta_3/q$. The desired statement can be deduced as in Case~2.

\medskip

Combining the three cases above, we may state that if  $p \in I\setminus J$,  $q\in J\setminus I$,  $0<\g_1,\g_2, \g\leq \infty$, $0<\theta<1$, and $1/r=(1-\theta)/p +\theta/q$, then
\[ 
A_{r,\g}=  (A_{p,\g_2}, A_{q,\g_2})_{\theta, \g},
\]
which is equivalent to  the statement  that  the family $(A_{p,\g})_{p, \g \in (0,\infty]}$ forms a real interpolation scale on  the interval $I\cup J$.
\end{proof}

\medskip

We now state the primary result of the paper. It is an extension of  \cite[Theorem~4.8]{Bekjan-Chen-Perrin-Y} to the full range $0<p<\infty$ and a noncommutative generalization of  \eqref{h-BMO} (see also \cite{Weisz3}).
 
\begin{theorem}\label{main-interpolation} If $0<\theta <1$, $0<p<\infty$, and  $0<\lambda, \g\leq \infty$,  then for $1/r=(1-\theta)/p$,
\[
\big( \h_{p,\lambda}^c(\M), \bmo^c(\M)\big)_{\theta, \g} = \h_{r,\g}^c(\M)
\]
with equivalent  quasi-norms.
\end{theorem}

Before we proceed, we should point out  that for $0<p<\infty$ and $0<\lambda \leq \infty$, the pair  $(\h_{p,\lambda}^c(\M), \bmo^c(\M))$  forms  a compatible couple. Indeed,  as described  in the preliminary
 section,  $\h_{p,\lambda}^c(\M)$ embeds isometrically into $L_{p,\lambda}(\M \overline{\otimes} B(\ell_2(\mathbb{N}^2)))$ which in turn is continuously embeded   into  the topological vector space $L_0(\M \overline{\otimes} B(\ell_2(\mathbb{N}^2)))$. Now we verify that $\bmo^c(\M)$ also embeds continuously into $L_0(\M \overline{\otimes} B(\ell_2(\mathbb{N}^2)))$.  This is immediate if $\M$ is finite as 
 $
 \bmo^c(\M) \subseteq \h_1^c(\M)  \subseteq L_0(\M \overline{\otimes} B(\ell_2(\mathbb{N}^2)))$.
 When $\M$ is infinite, this can be achieved  as follows:  since $\M_1$ is semifinite, choose a family of  mutually disjoint   finite projections $(e_j)_{j\in J}$ in $\M_1$ so that 
 $\sum_{j\in J} e_j={\bf 1}$ for the strong operator topology.
Let $x=(x_n)_{n\geq 1}  \in \bmo^c(\M)$ and for each $j\in J$, set $xe_j=(x_ne_j)_{n\geq 1}$. It is clear that $xe_j \in \bmo^c(\M)$ and $s_c^2(xe_j)=e_js_c^2(x) e_j$. Since $\T(e_j)<\infty$, $xe_j \in \h_2^c(\M)$. Therefore $U\cal{D}_c(x e_j)$ is well-defined in $L_0(\M \overline{\otimes} B(\ell_2(\mathbb{N}^2)))$. Moreover, the module property of $U$ implies that 
$U\cal{D}_c(xe_j) =U\cal{D}_c(xe_j) . (e_j \otimes Id)$ where $Id$ is the identity of 
$B(\ell_2(\mathbb{N}^2))$.  This implies that $\sum_{j\in J} U\cal{D}_c(xe_j)$ converges in measure. This shows in particular that the map $x \mapsto   \sum_{j\in J} U\cal{D}_c(xe_j)$  provides a continuous embedding of  $\bmo^c(\M)$ into $L_0(\M \overline{\otimes} B(\ell_2(\mathbb{N}^2)))$.
 
Similarly, if $0<r,s\leq \infty$,  then  $ \big( L_r^{\rm cond}(\M;\ell_2^c), L_s^{\rm cond}(\M;\ell_2^c)\big)$ is a compatible couple as both spaces embed continuously into 
$(L_r+L_s)^{\rm cond}(\M;\ell_2^c)$.

\bigskip

We need some preparation for the proof.  
We consider  the Orlicz  structure of the  symmetric function space $L_r +tL_s$ for  any given  $0<r<s<\infty$ and $t>0$.   To describe  this structure, we consider the following Orlicz function:
\begin{equation*}
\Phi_t^{(r,s)}(u)=\min\{u^r, t^su^s \} , \quad u\geq 0.
\end{equation*}

One can  verify that $\Phi_t^{(r,s)}(\cdot)$ is $r$-convex and $s$-concave. According to \cite[Lemma~3.2]{Hitc-SMS}, we have  for every $f\in L_r +L_s$ and $t>0$,  
\begin{equation}\label{norm-equiv}
2^{-1-1/r} \big\| f \big\|_{\Phi_t^{(r,s)}} \leq  K\big(f, t ; L_r, L_s\big) \leq 2 \big\| f \big\|_{\Phi_t^{(r,s)}}.
\end{equation}
 In particular,  the Orlicz space $L_{\Phi_t^{(r,s)}}$ coincides with  the space $L_r + t L_s$ with  isomorphism constant   depending only on the index $r$ and therefore  independent of $t$. 

Let $0<p,q,r,s<2$ and assume that $p<q$, $1/p=1/2 +1/r$, and $1/q=1/2 +1/s$.
One can easily check that  
\[
L_{p} + tL_{q}=L_2  \odot (L_{ r} + tL_{s}). 
\]
It follows  by identification  that  the following factorization holds for the corresponding  Orlicz spaces:
\[
L_{\Phi_t^{(p,q)}}= L_2 \odot L_{\Phi_t^{(r,s)}},
\]
with constants depending only on the indices $p$ and $q$. 

We will verify that the atomic decomposition results from  Theorem~\ref{main-algebraic} and Proposition~\ref{factorization-adapted} apply to this specific factorization. Ideally, we would like to have that the function $u\mapsto u^{-1/2}(\Phi_t^{(r,s)})^{-1}(u)$ is equivalent to $(\Phi_t^{(p,q)})^{-1}$ but  in order to avoid working with   inverse functions,  we will proceed directly  to  proving a corresponding result to Lemma~\ref{psi-prop}.  It also highlights the fact that for this particular case,  integral representations of the Orlicz functions involved  are  not needed.
To this end, set for $u >0$,
\[
\psi_t^{(p,q)}(u) := (u^{2-p} +t^{-q}u^{2-q})^{-1}.
\]
We have the following properties:
\begin{lemma}\label{psi-prop2}
\begin{enumerate}[{\rm(i)}]
\item $\psi_t^{(p,q)}(u) \approx_{p,q} u^{-2}\Phi_t^{(p,q)}(u)$;
\item $\Phi_t^{(r,s)} \big( (\psi_t^{(p,q)}(u))^{-1/2}\big) \lesssim_{p,q} \Phi_t^{(p,q)}(u)$;
\item $u \mapsto \psi_t^{(p,q)}(u)$ is operator monotone decreasing;
\item for any increasing sequence of  positive operators $a_n\uparrow a$, we have
\[
\sum_{n\geq 1} \T\big( (a_{n+1}^2 -a_n^2)\psi_t^{(p,q)}(a_{n+1})\big) \lesssim_{p,q} 
\T\big(\Phi_t^{(p,q)}(a)\big).
\]
\end{enumerate}
\end{lemma}
\begin{proof}
For the first item $(i)$, we write $u^{-2}\Phi_t^{(p,q)}(u)=\min\{u^{-2+p}, t^q u^{-2+q}\}=\big(\max\{u^{2-p}, t^{-q}u^{2-p}\}\big)^{-1}$  and note that 
\[
2^{-1}(u^{2-p} + t^{-q}u^{2-p}) \leq \max\{u^{2-p}, t^{-q}u^{2-p} \} \leq u^{2-p} + t^{-q}u^{2-p}.
\]
It follows  from taking inverses that 
\[
\psi_t^{(p,q)}(u) \leq u^{-2}\Phi_t^{(p,q)}(u)\leq 2 \psi_t^{(p,q)}(u) .
\]
Next,  item $(iii)$ can be seen as follows: if  $a$ and $b$ are positive operators with $a\leq b$ then since $0<1-(p/2) < 1$ and $0<1-(q/2)<1$, we have
\[
a^{1-(p/2)} +t^q a^{1-(q/2)} \leq  b^{1-(p/2)} +t^q b^{1-(q/2)}.
\]
Taking  inverse operators, we see that
\[
\psi_t^{(p,q)}(b^{1/2}) \leq \psi_t^{(p,q)}(a^{1/2}).
\]
This shows that $u\mapsto \psi_t^{(p,q)}(u^{1/2})$ is operator monotone decreasing. 

With the equivalence $(i)$ on hand, the proof of $(iv)$ is identical to the proof of \cite[Lemma~3.4]{Ran-Wu-Zhou} so we leave it to the reader. It remains to verify $(ii)$. From the equivalence $(i)$, it suffices to prove that 
\[
\Phi_t^{(r,s)} \big( u(\Phi_t^{(p,q)}(u))^{-1/2}\big) \lesssim_{p,q} \Phi_t^{(p,q)}(u).
\]
Since $u(\Phi_t^{(p,q)}(u))^{-1/2}= \max\{u^{1-{p/2}}, t^{-q/2} u^{1-{q/2}}\}$,  we have
from the definition of  $\Phi_t^{(r,s)}(\cdot)$ that
\begin{equation}\label{phi-pr}
\Phi_t^{(r,s)} \big( u(\Phi_t^{(p,q)}(u))^{-1/2}\big)=
\min\Big\{\max\{(u^{1-{p/2}})^r, (t^{-q/2} u^{1-{q/2}})^r\}, \max\{t^s(u^{1-{p/2}})^s, t^s(t^{-q/2} u^{1-{q/2}})^s\}\Big\}.
\end{equation}
Note that $(u^{1-{p/2}})^r =u^p$ and $t^s(t^{-q/2} u^{1-{q/2}})^s=t^qu^q$.  If $u^{1-(p/2)} \geq  t^{-q/2} u^{1-{q/2}}$ then   $t^qu^q \geq u^p$. In particular, $\Phi_t^{(p,q)}(u)=u^p$.  Then \eqref{phi-pr} implies that 
\[
\Phi_t^{(r,s)} \big( u(\Phi_t^{(p,q)}(u))^{-1/2}\big)=\min\big\{u^p, t^s (u^{1-(p/2)})^s\big\}\leq \Phi_t^{(p,q)}(u).
\]
Similarly, if $u^{1-(p/2)} \leq   t^{-q/2} u^{1-{q/2}}$ then  $t^qu^q \leq u^p$ and therefore, $\Phi_t^{(p,q)}(u)=t^qu^q$. We can deduce from \eqref{phi-pr} that
\[
\Phi_t^{(r,s)} \big( u(\Phi_t^{(p,q)}(u))^{-1/2}\big)=\min\Big\{(t^{-q/2} u^{1-{q/2}})^r, t^q u^q\Big\}\leq  \Phi_t^{(p,q)}(u).
\]
This completes the proof.
\end{proof}

\bigskip


 The next proposition constitutes the decisive step toward   our proof of Theorem~\ref{main-interpolation}. We formulate it here for the  more general conditioned spaces.

\begin{proposition}\label{Key-lemma} Let $\nu$ be a positive integer with $\nu\geq 2$.  Assume that $2/(\nu +1) <p\leq 2/\nu$ and $p <q<2/(\nu-1)$. 
If  $x \in \mathfrak{F}$,  then for  every $t>0$,
\[
K\big(x,t; L_p^{\rm cond}(\M;\ell_2^c), L_q^{\rm cond}(\M;\ell_2^c)\big) \approx_{p,q} K\big(\sigma_c(x), t; L_p(\M), L_q(\M)\big).
\]
\end{proposition}

\begin{proof}
$\bullet$
We observe first that one inequality in the   equivalence follows easily from the isometric embeddings described in the preliminary section. Indeed, let  $x \in \mathfrak{F}$  and  set $\cal{N}= \M \overline{\otimes} B(\ell_2(\mathbb N^2))$. Recall that for $0<r\leq \infty$, the map  $U: L_r^{\rm cond}(\M;\ell_2^c) \to L_r(\N)$ is an isometry  satisfying the identity
\[
|U(x)| = \sigma_c(x) \otimes e_{1,1} \otimes e_{1,1}.
\]
It follows that for every $t>0$,
\begin{align*}
K\big(\sigma_c(x), t; L_p(\M), L_q(\M)\big) &=K\big(|U(x)|, t; L_p(\cal{N}), L_q(\cal{N})\big)\\
&= K\big(U(x), t; L_p(\cal{N}), L_q(\cal{N})\big) \\
&\leq K\big(x,t; L_p^{\rm cond}(\M;\ell_2^c), L_q^{\rm cond}(\M;\ell_2^c)\big)
\end{align*}
which  verifies one inequality with constant $1$.

\smallskip

$\bullet$
The reverse inequality is more involved. 
 We consider the following two finite sequences of indices:  set $p_0=p$ and $q_0=q$ and for $1\leq m \leq \nu-1$,
\[
1/p_{m-1}=1/2 + 1/p_m  \ \text{and}\ 1/q_{m-1} =1/2 + 1/q_m.
\]
 From earlier discussions, the following factorization holds for the respective Orlicz spaces:
\[
L_{\Phi_t^{(p_{m-1},q_{m-1})}}= L_2 \odot L_{\Phi_t^{(p_m,q_m)}}.
\]
We note that  by Lemma~\ref{psi-prop2},    atomic decompositions stated in Theorem~\ref{main-algebraic} and Proposition~\ref{factorization-adapted} apply to each of these factorizations.

The main idea in the argument below is  to repeatedly apply the above factorization  until one gets indices that are strictly larger than $1$. At that  point,  splittings into adapted sequences are possible via the  noncommutative Stein inequality. The restriction imposed on the values of $p$ and $q$  is needed in the argument  since at every step (except the last one) we need  both   indices $p_m$ and $q_m$ to remain in the open  interval $(0\, ,\, 2)$ so that algebraic atomic  decompositions from the previous section  can be applied.

We  now present the details of the proof. Assume first that $x\in \mathfrak{F}^0$ (the general case will be dealt later).

Fix $t>0$. We apply Theorem~\ref{main-algebraic} to the Orlicz  space
$L_{\Phi_t^{(p,q)}}$.
We have the following  factorization:
\begin{equation}\label{factor1}
x=\a^{(1)}. \, \b^{(1)}
\end{equation}
where $\a^{(1)}$ is a strictly lower triangular matrix with entries in $L_2(\M)$ and  $\b^{(1)}$ is an adapted column  matrix in $L_{\Phi_t^{(p_1,q_1)}}(\M; \ell_2^c)$ satisfying: 
\begin{equation}\label{ab-1}
\big\|\a^{(1)}\big\|_2 \ .  \ \big\|\b^{(1)} \big\|_{L_{\Phi_t^{(p_1,q_1)}}(\M;\ell_2^c)} \lesssim_{p,q} \big\|\sigma_c(x) \big\|_{L_{\Phi_t^{(p,q)}}(\M)} .
\end{equation}
 Next, we  inductively construct   finite sequences  $\{\a^{(m)}: 2\leq  m \leq \nu-1 \}$  and 
 $\{ \b^{(m)} : 1\leq m\leq \nu-1\}$ where the $\a^{(m)}$'s are lower triangular matrices  taking values in $L_2(\M)$  satisfying $\a^{(m)}_{n,j} \in L_2(\M_n)$ for $1\leq j\leq n$, and $\b^{(m)}$'s are adapted column matrices. Both sequences satisfy   for $2\leq m\leq \nu-1$:

\begin{equation}\label{factor2}
\b^{(m-1)}  =\a^{(m)}\,  . \,  \b^{(m)}
\end{equation}
and
 \begin{equation}\label{ab}
 \big\|\a^{(m)}\big\|_2\  . \ \big\|\b^{(m)} \big\|_{L_{\Phi_t^{(p_m,q_m)}}(\M;\ell_2^c)} 
 \lesssim_{p,q} \big\| \b^{(m-1)} \big\|_{L_{\Phi_t^{(p_{m-1},q_{m-1})}}(\M; \ell_2^c)}.
 \end{equation}
This is done by applying Proposition~\ref{factorization-adapted}  to $\b^{(m-1)} \in  
L_{\Phi_t^{(p_{m-1},q_{m-1})}}^{\rm ad}(\M; \ell_2^c)$. The constant   depends on $p_{m-1}$ and $q_{m-1}$ but since they depend on $p$ and $q$ respectively, we may state that for each step, the  constant depends on $p$ and $q$.

 Clearly, the above construction induces a factorization:
 \[
 x=\a^{(1)}\dots \a^{(\nu-1)}\,  . \,  \b^{(\nu-1)}.
 \]
 Now, we consider the adapted sequence  $\b^{(\nu-1)}\in L_{\Phi_t^{(p_{\nu-1}, q_{\nu-1})}}(\M;\ell_2^c)$. By identification, we have $\b^{(\nu-1)} \in (L_{p_{\nu-1}} +t L_{q_{\nu-1}})(\M;\ell_2^c)=L_{p_{\nu-1}}(\M;\ell_2^c) + t L_{q_{\nu-1}}(\M;\ell_2^c)$.
 
 It follows  from the  definition of sum of two Banach spaces that  $\b^{(\nu-1)}$ admits a decomposition $\b^{(\nu-1)}=\xi^{(1)} +\xi^{(2)}$ with $\xi^{(1)} \in L_{p_{\nu-1}}(\M;\ell_2^c)$ and  $ \xi^{(2)} \in L_{q_{\nu-1}}(\M;\ell_2^c)$ satisfying the  norm estimate: 
\begin{equation}\label{b-splitting}
\big\| \xi^{(1)} \big\|_{L_{p_{\nu-1}}(\M;\ell_2^c)} +t\big\| \xi^{(2)} \big\|_{L_{q_{\nu-1}}(\M;\ell_2^c)} \leq 2\big\|\b^{(\nu-1)} \big\|_{(L_{p_{\nu-1}} + tL_{q_{\nu-1}})(\M;\ell_2^c)}.
\end{equation}
The important fact here is that $1<p_{\nu-1} <q_{\nu-1} <\infty$.
By applying the noncommutative Stein inequality (\cite{PX}), we may replace $\xi^{(1)}$ and $\xi^{(2)}$ by  adapted sequences  $\zeta^{(1)}=\{\E_n(\xi_n^{(1)})\}_{n\geq 1}$ and $\zeta^{(2)}=\{\E_n(\xi_n^{(2)})\}_{n\geq 1}$ satisfying:
\begin{equation}\label{g-splitting}
\big\| \zeta^{(1)} \big\|_{L_{p_{\nu-1}}(\M;\ell_2^c)} +t\big\| \zeta^{(2)} \big\|_{L_{q_{\nu-1}}(\M;\ell_2^c)} \leq C_{p,q}\big\|\b^{(\nu-1)} \big\|_{(L_{p_{\nu-1}} + tL_{q_{\nu-1}})(\M;\ell_2^c)}. 
\end{equation}
The constant $C_{p,q}$ can be taken to be equal to $2\max\{\g_{p_{\nu-1}}, \g_{q_{\nu-1}}\}$ where $\g_r$ is the constant from the noncommutative Stein inequality for $1<r<\infty$. This justifies that $C_{p,q}$  depends only on $p$ and $q$ since $p_{\nu-1}$ and $q_{\nu-1}$ depend on $p$  and $q$ respectively.

Next, we consider  two sequences of operators by setting:
\[
x^{(1)} =\a^{(1)}\dots \a^{(\nu-1)}.\ \zeta^{(1)}   \quad \text{and} \quad 
x^{(2)} =\a^{(1)}\dots \a^{(\nu-1)}. \ \zeta^{(2)}.
\]
Since $\b^{(\nu-1)}= \zeta^{(1)} +\zeta^{(2)}$, we clearly have  $x =x^{(1)} + x^{(2)}$. We claim  that
 $x^{(1)} \in L_p^{\rm cond}(\M;\ell_2^c)$ and $x^{(2)} \in L_q^{\rm cond}(\M;\ell_2^c)$.   Indeed, let $y^{(1)}=\a^{(2)}\dots \a^{(\nu-1)}. \ \zeta^{(1)}$ and $y^{(2)}=\a^{(2)}\dots \a^{(\nu-1)}. \ \zeta^{(2)}$. The fact that $\zeta^{(1)}$ is adapted and the property  that $\a_{n,j}^{(\nu-1)} \in L_2(\M_n)$ for $1\leq j  \leq n$ implies that 
 $\widehat{\b}^{(\nu-2)}:=\a^{(\nu-1)}\, . \, \zeta^{(1)}$ is an adapted sequence. Moreover, by H\"older's inequality,  we have  $\widehat{\b}^{(\nu-2)} \in L_{p_{\nu-2}(\M;\ell_2^c)}$ with 
 \[
 \big\| \widehat{\b}^{(\nu-2)}\big\|_{L_{p_{\nu-2}}(\M;\ell_2^c)} \leq \big\|\a^{(\nu-1)}\big\|_2 \, . \, \big\| \zeta^{(1)} \big\|_{L_{p_{\nu-1}}(\M;\ell_2^c)}.
 \]
 One  can work backward  and set $\widehat{\b}^{(m-1)} = \a^{(m)}\, .\, \widehat{\b}^{(m)}$ for $2\leq m \leq \nu-2$ 
 to see that $y^{(1)}$ is an adapted column  sequence. Moreover, a repeated use of  H\"older's inequality yields:
 \[
 \big\|y^{(1)}\big\|_{L_{p_1}(\M;\ell_2^c)} \leq \Big(\prod_{m=2}^{\nu-1}  \big\|\a^{(m)}\big\|_2\Big) \big\| \zeta^{(1)}\big\|_{L_{p_{\nu-1}}(\M;\ell_2^c)} <\infty.
 \]
Since $\a^{(1)}$ is strictly lower triangular and $y^{(1)}$ is adapted, it follows that $x^{(1)}= \a^{(1)}\, . \, y^{(1)} \in L_p^{\rm cond}(\M;\ell_2^c)$. Similar argument can be applied to  deduce that  $x^{(2)} \in L_q^{\rm cond}(\M;\ell_2^c)$.  

We now verify that the decomposition $x= x^{(1)} + x^{(2)}$ provides the desired inequality between  the  two $K$-functionals. Indeed, 
 we have  the following norm estimates:
\begin{align*}
K\big(x,t; L_p^{\rm cond}(\M;\ell_2^c), L_q^{\rm cond}(\M;\ell_2^c)\big) &\leq
\big\|x^{(1)} \big\|_{L_p^{\rm cond}(\M;\ell_2^c)} +t \big\|x^{(2)} \big\|_{L_q^{\rm cond}(\M;\ell_2^c)} \\
 &\leq  \Big(\prod_{m=1}^{\nu-1} \big\|\a^{(m)}\big\|_2\Big). \big( \big\|\zeta^{(1)}\big\|_{L_{p_{\nu-1}}(\M;\ell_2^c)} + t\big\|\zeta^{(2)}\big\|_{L_{q_{\nu-1}}(\M;\ell_2^c)} \big)\\
 &\lesssim_{p,q}\Big(  \prod_{m=1}^{\nu-1} \big\|\a^{(m)}\big\|_2\Big) . \, \big\|\b^{(\nu-1)} \big\|_{(L_{p_{\nu-1}} + tL_{q_{\nu-1}})(\M;\ell_2^c)}\\
&\lesssim_{p,q} \Big(\prod_{m=1}^{\nu-1} \big\|\a^{(m)}\big\|_2 \Big) .  \, \big\|\b^{(\nu-1)} \big\|_{L_{\Phi_t^{(p_{\nu-1}, q_{\nu-1})}}(\M;\ell_2^c)}.
\end{align*}
Applying \eqref{ab} successively $(\nu-2)$-times, we get
\[
K\big(x,t; L_p^{\rm cond}(\M;\ell_2^c), L_q^{\rm cond}(\M;\ell_2^c)\big) \lesssim_{p,q} \big\|\a^{(1)}\big\|_2\,  .  \, \big\|\b^{(1)} \big\|_{L_{\Phi_t^{(p_{1}, q_{1})}}(\M;\ell_2^c)}.
\]
By \eqref{ab-1} and  norm equivalence, we arrive at
\[
K\big(x,t; L_p^{\rm cond}(\M;\ell_2^c), L_q^{\rm cond}(\M;\ell_2^c)\big) \lesssim_{p,q}  \big\| \sigma_c(x) \big\|_{(L_p +tL_q)(\M)}.
\]
This   proves the case  where $x \in \mathfrak{F}^0$.

\smallskip

We now consider  an arbitrary  $x=(x_n)_{n\geq 1}  \in \mathfrak{F}$.  Let    $z=(0,x_2, x_3,\dots) \in \mathfrak{F}^0$. By the direct sum \eqref{direct-sum} and the previous case, 
  we have:
\begin{align*}
K\big(x,t; &L_p^{\rm cond}(\M;\ell_2^c), L_q^{\rm cond}(\M;\ell_2^c)\big)\\ &\lesssim_{p,q} 
K\big(x_1, t; L_p^c(\M, \E_1), L_q^c(\M,\E_1) \big)+K\big(z,t; L_p^{,\rm cond}(\M;\ell_2^c), L_q^{\rm cond}(\M;\ell_2^c)\big)\\
&\lesssim_{p,q} K\big(x_1, t; L_p^c(\M,\E_1), L_q^c(\M, \E_1) \big) +K\big(\sigma_c(z), t; L_p(\M), L_q(\M) \big).
\end{align*}

Since $u_{1}(L_p^c(\M,\E_1))$ and $u_{1}(L_q^c(\M,\E_1))$  are one complemented in $L_p(\M_1;\ell_2^c)$ and $L_q(\M_1;\ell_2^c)$ respectively, we have 
\begin{align*}
K\big(x_1, t; L_p^c(\M, \E_1), L_q^c(\M,\E_1) \big) &=K\big(u_1(x_1), t; L_p(\M_1; \ell_2^c), L_q(\M_1;\ell_2^c) \big)\\
&=K\big(|u_1(x_1)|, t; L_p(\M_1; \ell_2^c), L_q(\M_1;\ell_2^c) \big).
\end{align*}
As $| u_1(x_1)| =\E_1(|x_1|^2)^{1/2} \otimes e_{1,1}$,  it follows that 
\[
K\big(x_1, t; L_p^c(\M, \E_1), L_q^c(\M,\E_1) \big) = K\big(\E_1(|x_1|^2)^{1/2}, t; L_p(\M_1), L_q(\M_1) \big).
\]
As  $\E_1(|x_1|^2)^{1/2} \leq \sigma_c(x)$ and $\sigma_c(z)\leq \sigma_c(x)$, we  may conclude that
\[
K\big(x,t; L_p^{\rm cond}(\M;\ell_2^c), L_q^{\rm cond}(\M;\ell_2^c)\big) \lesssim_{p,q} 
K\big(\sigma_c(x), t; L_p(\M), L_q(\M) \big).
\]
The proof is complete.
\end{proof}

Using the fact that for every $0<r \leq \infty$,   $\h_r^c(\M)$  embeds isometrically into a   $1$-complemented subspace of  $L_r^{\rm cond}(\M;\ell_2^c)$, the next result follows     immediately   from Proposition~\ref{Key-lemma}. 

\begin{corollary}\label{Key-lemma-h} Let $\nu$ be a positive integer with 
$\nu\geq 2$. Assume that $2/(\nu +1) <p\leq 2/\nu$ and $p <q< 2/(\nu-1)$. 
If  $y$ is a finite martingale  in $\mathfrak{F}(M)$ then for  every $t>0$,
\[
K\big(y,t; \h_p^c(\M), \h_q^c(\M)\big) \approx_{p,q} K\big(s_c(y), t; L_p(\M), L_q(\M)\big).
\]
\end{corollary}

\begin{remark}\label{gen-Key-lemma} Since  $\mathfrak{F}$ is  dense in $L_p^{\rm cond}(\M;\ell_2^c) +L_q^{\rm cond}(\M;\ell_2^c)$, we have  the following more general  assertion that under the assumption of Proposition~\ref{Key-lemma},  for every $x \in L_p^{\rm cond}(\M;\ell_2^c) + L_q^{\rm cond}(\M;\ell_2^c)$ and $t>0$,
\begin{equation*}
K\big(x,t; L_p^{\rm cond}(\M;\ell_2^c), L_q^{\rm cond}(\M;\ell_2^c)\big) \approx_{p,q}
K\big(U(x) , t; L_p(\N), L_q(\N)\big)
\end{equation*}
where $\N=\M \overline{\otimes} B(\ell_2(\mathbb N^2))$.

Similarly, since $\mathfrak{F}(M)$ is dense in $\h_p^c(\M) + \h_q^c(\M)$, it follows that  for every $y \in \h_p^c(\M) + \h_q^c(\M)$ and $t>0$,
\begin{equation*}
K\big(y,t; \h_p^c(\M), h_q^c(\M)\big) \approx_{p,q}
K\big(U\cal{D}_c(y) , t; L_p(\N), L_q(\N)\big).
\end{equation*}
\end{remark}

%

\subsection{Proof of Theorem~\ref{main-interpolation}}
We consider first   two important  intermediate cases. One is  the Banach space range and the other is  when the distance between the two indices is small enough. We begin with the latter.
\begin{lemma}\label{main-Step}
Let  $\nu_0 \geq 2$. Consider  $2/(\nu_0+1) <p_0 \leq 2/\nu_0$ and $p_0<p_1 <2/(\nu_0-1)$.  If $0<\theta<1$, $0<\g_0,\g_1, \g \leq \infty$, and $1/r=(1-\theta)/p_0 + \theta/p_1$, then
\begin{equation*}
\h_{r, \g}^c(\M)= \big( \h_{p_0,\g_0}^c(\M) , \h_{p_1,\g_1}^c(\M)\big)_{\theta, \g}
\end{equation*}
with equivalent quasi-norms.
\end{lemma}
\begin{proof}
This will be  deduced  from Corollary~\ref{Key-lemma-h} (see also Remark~\ref{gen-Key-lemma}) and the description of noncommutative Lorentz spaces as real interpolation of  the couple $(L_{p_0}(\M), L_{p_1}(\M))$.  Indeed,  if $y \in \h_{r,\g}^c(\M)$, then  we have:
\begin{align*}
\big\| y \big\|_{\h_{r,\gamma}^c}&=\big\| U\cal{D}_c(y)\big\|_{L_{r,\g}(\N)}\\
 &\approx \big\| U\cal{D}_c(y)\big\|_{(L_{p_0}(\N), L_{p_1}(\N))_{\theta, \gamma} }\\
&\approx \big\| y \big\|_{(\h_{p_0}^c(\M),\h_{p_1}^c(\M))_{\theta,\g}}
\end{align*}
where the last equivalence comes from the comparison of $K$-functionals stated in Remark~\ref{gen-Key-lemma}.
 This  clearly shows that
$\h_{r, \g}^c(\M)= \big( \h_{p_0}^c(\M) , \h_{p_1}^c(\M)\big)_{\theta, \g}$. The full generality as stated in the lemma  follows by reiteration.
\end{proof}


The next lemma is the infinite version of the Banach space case.
It will be deduced from Lemma~\ref{main-Step} and Wolff's interpolation theorem.  Since our approach differs from \cite{Bekjan-Chen-Perrin-Y}, we include the details.  
\begin{lemma}\label{banach}
Let $1<p<r<\infty$. 
 If $1/r=(1-\theta)/p$  then 
\begin{equation*}
\h_{r, \g}^c(\M)= \big( \h_p^c(\M) , \bmo^c(\M)\big)_{\theta, r}
\end{equation*}
with equivalent norms.
\end{lemma}
\begin{proof}  We will verify first that if $1<r<q<\infty$ then 
\begin{equation}\label{1-q}
\h_r^c(\M)=\big(\h_1^c(\M), \h_q^c(\M)\big)_{\psi, r}
\end{equation}
for $1/r=(1-\psi) +\psi/q$. We separate the proof of \eqref{1-q}  into three  cases:

$\bullet$ $1<r<q<2$. This  follows  immediately from  
 using $\nu_0=2$,  $p_0=1$, and $p_1=q$  in Lemma~\ref{main-Step}.

$\bullet$ $1<r<2\leq q<\infty$. Fix $v$ such that $1<r<v<2\leq q$. Applying the previous case, we have 
\[
\h_r^c(\M)=\big(\h_1^c(\M), \h_v^c(\M)\big)_{\psi_1, r}
\]
for $1/r=(1-\psi_1) +\psi_1/v$. On the other hand, since for every $1<s<\infty$, $\h_s^c(\M)$ embeds complementably into $L_s(\M \overline{\otimes} B(\ell_2(\mathbb{N}^2)))$, we have for
$1<r<v< q<\infty$  and $1/v= (1-\psi_2)/r + \psi_2/q$ that 
\[
\h_v^c(\M)=\big(\h_r^c(\M), \h_q^c(\M)\big)_{\psi_2, v}.
\]

 Set $B_1=\h_1^c(\M)$,  $B_2=\h_r^c(\M)$, $B_3=\h_v^c(\M)$, and $B_4=\h_q^c(\M)$. It is clear  that $B_1\cap B_4$ is dense in both $B_2$ and $B_3$. Applying Wolff's interpolation theorem,
we  deduce \eqref{1-q}   with $\psi=\psi_1\psi_2(1-\psi_2 +\psi_1\psi_2)^{-1}$. One  can easily verify that this is  the desired index.

$\bullet$  $2\leq r<q<\infty$. Fix $1<u <2 \leq r<q$ and write
\[
\h_r^c(\M)=\big(\h_u^c(\M), \h_q^c(\M)\big)_{\theta_1, r}.
\]
Next, we have from the previous case that
\[
\h_u^c(\M)=\big(\h_1^c(\M), \h_q^c(\M)\big)_{\theta_2, u}.
\]
Applying Wolff's interpolation theorem with  $B_1=\h_1^c(\M)$, $B_3=\h_r^c(\M)$, $B_2=\h_u^c(\M)$, and $B_4= h_q^c(\M)$, the desired interpolation follows with $\psi= \theta_1(1-\theta_2 +\theta_2\theta_1)^{-1}$. This proves \eqref{1-q}.

Using  the  description of the dual of $\h_1^c(\M)$ from  \eqref{h1-dual} and the well-known  fact that $(\h_v^c(\M))^*=\h_{v'}^c(\M)$ for $1<v<\infty$ and $v'$ is its conjugate index  (\cite{JX}),  we obtain from \eqref{1-q} and the duality for interpolation (\cite[Theorem~3.7.1]{BL})  that if $1<p<r<\infty$, then 
\[
\h_r^c(\M) = \big(\h_p^c(\M), \bmo^c(\M)\big)_{\theta, r}
\]
which is the desired conclusion.
\end{proof}


\medskip

We are now ready to present the proof of Theorem~\ref{main-interpolation}.
 
For $0<p\leq \infty$, 
let $A_{p,\g} := \h_{p,\g}^c(\M)$ when  $0<p<\infty$  and $A_{\infty, \g}:=\bmo^c(\M)$.
Consider the sequence of intervals $(I_\nu)_{\nu\geq 1}$ with $I_1=(1,\infty]$ and 
for $\nu\geq 2$, 
\[
I_\nu= (\frac{2}{\nu+1}, \frac{2}{\nu-1}).
\]
 For a given $ \nu\geq 2$, it follows from Lemma~\ref{main-Step} that the family $\{ A_{p,\g}\}_{p,\g \in (0,\infty]}$ forms a real-interpolation scale on the interval $I_\nu$. On the other hand, by reiteration, Lemma~\ref{banach} gives that the family $\{ A_{p,\g}\}_{p,\g \in (0,\infty]}$ forms a real-interpolation scale on the interval $I_1$.

Next, we have $I_1 \cap I_2= (1,2)$ and for $\nu \geq 2$, $I_\nu \cap I_{\nu+1}= (2/(\nu+1), 2/\nu]$. By  applying Lemma~\ref{union} inductively,  we deduce that the family $\{ A_{p,\g}\}_{p,\g \in (0,\infty]}$ forms a real-interpolation scale on the interval $\bigcup_{\nu=1}^\infty I_\nu = (0,\infty)$ which is the desired conclusion.
\qed

\medskip

Adapting the argument above by using Proposition~\ref{Key-lemma} in place of Corollary~\ref{Key-lemma-h}, we also obtain the corresponding result at the level of conditioned spaces.

\begin{proposition}\label{interpolation-conditioned}
If $0<\theta <1$, $0<p, q<\infty$, and  $0<\g\leq \infty$,  then for $1/{r}=(1-\theta)/{p} + {\theta}/q$, 
\[
\big( L_p^{\rm cond}(\M; \ell_2^c), L_q^{\rm cond}(\M; \ell_2^c)\big)_{\theta, \g} = L_{r,\g}^{\rm cond}(\M;\ell_2^c),
\]
with equivalent  quasi-norms.
\end{proposition}

\subsection{Spaces of adapted sequences}
In this subsection, we apply ideas used for the case of conditioned Hardy spaces to the family of spaces of adapted sequences. We first observe that
by the noncommutative Stein inequality (\cite{PX}), the space of adapted sequences $L_p^{\rm ad}(\M;\ell_2^c)$ is complemented in $L_p(\M;\ell_2^c)$ when $1<p<\infty$. Thus,  it is rather an easy task to see that the family $\{ L_p^{\rm ad}(\M;\ell_2^c)\}_{1<p<\infty}$ forms interpolation scales.
The next result extends  this fact to the full range $0<p<\infty$. 

\begin{theorem}\label{interpolation-adapted}
If $0<\theta <1$, $0<p, q<\infty$, and  $0<\g\leq \infty$,  then for $1/{r}=(1-\theta)/{p} + {\theta}/q$, 
\[
\big( L_p^{\rm ad}(\M; \ell_2^c), L_q^{\rm ad}(\M; \ell_2^c)\big)_{\theta, \g} = L_{r,\g}^{\rm ad}(\M;\ell_2^c),
\]
with equivalent quasi-norms.
\end{theorem}

As in the case of conditioned spaces, the proof is based on estimates of $K$-functionals. We observe first that since $L_r(\M;\ell_2^c)$ is a $1$-complemented subspace of $L_r(\M \overline{\otimes} B(\ell_2))$ for every $0<r<\infty$, one can easily compute the $K$-functionals for the couple $(L_p(\M;\ell_2^c),L_q(\M;\ell_2^c))$. Adapting the argument used in  the proof of Proposition~\ref{Key-lemma} (using Proposition~\ref{factorization-adapted} in place of Theorem~\ref{main-algebraic} in the first step), we obtain the corresponding result for  spaces of adapted sequences. More precisely:

\begin{proposition}\label{key-lemma-ad} Let $\nu$ be a positive integer with $\nu\geq 2$. Assume that $2/(\nu +1) <p\leq 2/\nu$ and $p <q <2/(\nu-1)$. 
For every sequence $a \in  L_p^{\rm ad}(\M;\ell_2^c)+ L_q^{\rm ad}(\M;\ell_2^c)$ and  every $t>0$, the following holds:
\[
K\big(a,t; L_p^{\rm ad}(\M;\ell_2^c), L_q^{\rm ad}(\M;\ell_2^c)\big) \approx_{p,q} K\big( \cal{S}_c(a), t; L_p(\M), L_q(\M)\big).
\]
\end{proposition}

\begin{proof}[Sketch of the proof  of Theorem~\ref{interpolation-adapted}] First, we use Proposition~\ref{key-lemma-ad} to deduce the corresponding result to Lemma~\ref{main-Step}.
Fix $\nu_0 \geq 2$. Assume that  $2/(\nu_0+1) <p_0 \leq 2/\nu_0$ and $p_0<p_1 <2/(\nu_0-1)$.  If $0<\theta<1$, $0<\g_0,\g_1, \g \leq \infty$, and $1/r=(1-\theta)/p_0 + \theta/p_1$, then
\begin{equation}\label{adapted-1}
L_{r, \g}^{\rm ad}(\M;\ell_2^c)= \big( L_{p_0,\g_0}^{\rm ad}(\M;\ell_2^c) , L_{p_1,\g_1}^{\rm ad}(\M ; \ell_2^c)\big)_{\theta, \g}.
\end{equation}
Next,  we deduce the  Banach space range using complementation:
if  $1< p<r  <q<\infty$ and $1/r=(1-\psi)/p + \psi/q$,  then 
\begin{equation}\label{adapted-2}
L_r^{\rm ad}(\M;\ell_2^c) =\big( L_{p}^{\rm ad}(\M;\ell_2^c) , L_{q}^{\rm ad}(\M:\ell_2^c)\big)_{\psi, r}.
\end{equation} 
Using  \eqref{adapted-1} and \eqref{adapted-2}, we can repeat the inductive argument used in the proof of Theorem~\ref{main-interpolation} with   the intervals $I_1=(1,\infty)$
and $I_\nu=({2}/(\nu +1), {2}/(\nu-1))$ for $\nu\geq 2$,
to conclude that the family  $\{L_p^{\rm ad}(\M;\ell_2^c)\}_{p\in(0, \infty); \g \in (0,\infty]}$ forms a real interpolation scale.
 \end{proof}

Assume that $1\leq  p<\infty$. We recall that the map $\cal{D}(x)= (dx_n)_{n\geq 1}$ is an isometric embedding of $\H_p^c(\M)$ into  $L_p^{\rm ad}(\M;\ell_2^c)$. 
Using  the noncommutative Stein inequality  when  $1<p<\infty$ and the noncommutative L\'epingle-Yor  inequality when $p=1$ (\cite{Qiu1}), the linear map 
\[
\Pi \big( (a_n)_{n\geq 1}\big)= (a_n-\E_{n-1}(a_n))_{n\geq 1}
\]
is  simultaneously bounded from $L_p^{\rm ad}(\M;\ell_2^c)$ onto $\cal{D}\big(\H_p^c(\M)\big)$ for all $1\leq p<\infty$.
   As a result, one can  immediately deduce from Theorem~\ref{interpolation-adapted}  that  the interpolation 
   \[(\H_1^c(\M), \H_q^c(\M))_{\theta, r} = \H_r^c(\M)
   \]
    for $0<\theta<1$, $1<q<\infty$,  and $ 1/r=(1-\theta) +\theta/q$ holds.  Due to  this fact, we may view Theorem~\ref{interpolation-adapted} as an extension of  the real  interpolation version of \cite{Musat-inter} to the quasi-Banach space range.
However, when 
 $0<p<1$,   the spaces of adapted  sequences cannot be replaced by column  martingale Hardy spaces. In fact, only  one inclusion holds if  column martingale Hardy spaces are used. The next  result may be viewed as a    noncommutative  generalization of 
\cite[Theorem~2]{Jason-Jones}. We  also direct the reader to  its companion Corollary~\ref{inclusion-complex}  below  and a discussion on the non validity  of the reverse inclusions.  We refer to \cite{Musat-inter, PX} for the definition of $\BMO^c(\M)$.

\begin{corollary}\label{inclusion-real}
Assume that   $0<p<1$ and $0<\theta<\infty$ are such that $1/r=(1-\theta)/p <1$.
Then 
\[
\big( \H_p^c(\M), \BMO^c(\M) \big)_{\theta, r}\subset \H_r^c(\M).
\]
\end{corollary}
\begin{proof}
Let  $0<\eta<1$ such that $1= (1-\eta)/p +\eta/r$. Let $ x\in \big(\H_p^c(\M),  \H_r^c(\M) \big)_{\eta, 1}$.  Since $\H_u^c(\M) \subset L_u^{\rm ad}(\M;\ell_2^c)$ for $u\in\{p,r\}$, we have 
\begin{align*}
\|x\|_{\H_1^c}  &=\| dx\|_{L_1^{\rm ad}(\M;\ell_2^c)}\\
&\approx  \| dx \|_{(L_p^{\rm ad}(\M;\ell_2^c), L_r^{\rm ad}(\M;\ell_2^c))_{\eta, 1}}\\
&\leq \|x \|_{(\H_p^c(\M),  \H_r^c(\M))_{\eta, 1}}
\end{align*}
where the second equivalence  comes from Proposition~\ref{interpolation-adapted}.
This shows that
\[
\big( \H_p^c(\M), \H_r^c(\M) \big)_{\eta, 1}\subset \H_1^c(\M).
\]
Next,   we recall that  $\big(\H_1^c(\M), \BMO^c(\M)\big)_{\phi, r} =\H_r^c(\M)$  where  
$1/r=1-\phi$.   We remark that the  Wolff's interpolation  theorem  is valid at the level  of inclusion: if  $B_1 \cap B_4 \subset B_2 \cap  B_3$,   $(B_1, B_3)_{\eta, q_1} \subset B_2$, and $(B_2,  B_4)_{\phi, q_2} \subset B_3$, then $(B_1, B_4)_{\theta, q_3} \subset  B_3$ whenever $0<q_j\leq \infty$ ($j=1,2,3$) and  $\theta=\phi/(1-\eta +\eta\phi)$. The verification of this fact can be found in  the first part of the  proof of \cite[Theorem~1]{Wolff}.  Using $B_1=\H_p^c(\M)$,  $B_2=\H_1^c(\M)$, $B_3= \H_r^c(\M)$, and $B_4=\BMO^c(\M)$, we obtain  the desired conclusion.
\end{proof}



\subsection{The complex method}

We now turn our attention to the case of complex interpolation method which we now briefly review.

Let $S$ (respectively, $\overline{S}$) denote the open strip $\{z : 0< \Re z<1\}$ (respectively, the closed strip $\{z : 0\leq \Re z \leq 1\}$) in the complex plane $\C$. Let $A(S)$ be the  
collection of $\C$-valued functions that are analytic on $S$ and continuous and bounded  on $\overline{S}$.  For a compatible couple of complex quasi-Banach spaces $(A_0,A_1)$, we denote by $\cal{F}_0(A_0,A_1)$ the family of functions of the form $f(z)= \sum_{k=1}^n f_k(z)x_k$ with  $f_k \in A(S)$ and  $x_k \in A_0 \cap A_1$.
We equip $\cal{F}_0(A_0,A_1)$ with the quasi-norm:
\[
\big\|f\big\|_{\cal{F}_0(A_0,A_1)}= \max \Big\{ \sup_{t\in \R} \big\| f(it)\big\|_{A_0} , \sup_{t\in \R} \big\| f(1+it)\big\|_{A_1}\Big\}.
\]
Then $\cal{F}_0(A_0,A_1)$  becomes a quasi-Banach space. For $0<\theta<1$, the complex  interpolation norm on $A_0 \cap A_1$ is defined by:
\[
\big\|x \big\|_{[A_0,A_1]_\theta} =\inf\Big\{\big\|f\big\|_{\cal{F}_0(A_0,A_1)} : f(\theta)=x,\  f \in \cal{F}_0(A_0,A_1)\Big\}.
\]
The \emph{complex interpolation space} (of exponent $\theta$)  $[A_0,A_1]_\theta$ is defined as the completion of the quasi-normed space $(A_0 \cap A_1, \| \cdot \|_{[A_0,A_1]_\theta})$.

As in the real method, complex interpolations of the couple $(L_p, L_q)$ (for $0<p<q\leq \infty$)  are well-known.  Indeed, if $\N$ is a semifinite von Neumann algebra and $0<\theta<1$, then  
\[
\big[L_p(\N), L_q(\N)\big]_\theta =L_r(\N)
\]
 isometrically for $1/r=(1-\theta)/p +\theta/q$. This fact comes from \cite[Theorem~4.1]{Xu-inter} (see also  \cite{PX3}).

The next result   is the version of Theorem~\ref{main-interpolation} for the  complex method.
\begin{theorem}\label{main-complex}
 If  $0<p,q<\infty$,  $0<\theta<1$,  and  $\displaystyle{\frac{1}{r}=\frac{1-\theta}{p} + \frac{\theta}{q}}$,   then
 \[
 \h_r^c(\M) =\big[\h_p^c(\M), \h_q^c(\M) \big]_\theta
 \]
with equivalent  quasi-norms.
\end{theorem}

For the proof, we will use the next result which provides  a connection between complex interpolation method and real interpolation method  that is valid  for quasi-Banach spaces. We should note that for Banach spaces, the inequality in the next theorem is actually an equivalence but  for quasi-Banach spaces only one inequality is valid in its full generality.
\begin{theorem}[{\cite[Theorem~3]{Cw-Mil-Sag}}]\label{connection}
Let $(A_0, A_1)$ be a compatible couple of quasi-Banach spaces. Let 
$0<\theta_j<1$, $0<\theta <1$, and $0<\g_j\leq \infty$. Denote $E_j= (A_0, A_1)_{\theta_j, \g_j}$ for $j=0,1$.  If $1/\g= (1-\theta)/\g_0 +\theta/\g_1$ and $\lambda=(1-\theta)\theta_0 +\theta\theta_1$  then  for every $a \in A_0\cap  A_1\subset E_0 \cap E_1$,
\[
\big\| a \big\|_{[E_0,E_1]_\theta} \leq C \big\|a \big\|_{(A_0, A_1)_{\lambda, \g}}.
\]
\end{theorem}

\begin{proof}[Proof of Theorem~\ref{main-complex}] Assume that $0<p<q<\infty$ and fix $0<p_0<p$. Then according to Theorem~\ref{main-interpolation}, $\h_p^c(\M)=(\h_{p_0}^c(\M), \bmo^c(\M))_{\theta_0, p}$ and \ $h_q^c(\M)=(\h_{p_0}^c(\M), \bmo^c(\M))_{\theta_1, q}$ where $1/p=(1-\theta_0)/p_0$ and $1/q=(1-\theta_1)/p_0$. We may state from Theorem~\ref{connection} that  if $1/\g= (1-\theta)/p +\theta/q$ and 
$\lambda=(1-\theta)\theta_0 +\theta\theta_1$  then for every $a\in \h_{p_0}^c(\M) \cap \bmo^c(\M)$,
\[
\big\| a \big\|_{[\h_p^c(\M), \h_q^c(\M)]_\theta} \leq C \big\|a\big\|_{(\h_{p_0}^c(\M), \bmo^c(\M))_{\lambda, \gamma}}.
\] 
One can easily verify that $(1-\lambda)/{p_0}=1/\g=1/r$ and therefore we may deduce  from Theorem~\ref{main-interpolation} that
\begin{equation}\label{right}
\big\| a \big\|_{[\h_p^c(\M), \h_q^c(\M)]_\theta} \leq C' \big\|a\big\|_{\h_r^c}.
\end{equation}
On the other hand, let $U: L_s^{\rm cond}(\M;\ell_2^c) \to L_s(\cal N)$ (where $\cal{N}=\M\overline{\otimes} B(\ell_2(\mathbb{N}^2))$) be the family of  isometric embeddings  as described in the preliminary section which  are valid for all $0<s\leq \infty$. Denote by $\cal{D}_c$ the extension of the map $x\mapsto  (dx_n)_{n\geq 1}$
from $\h_s^c(\M)$ into  $L_s^{\rm cond}(\M;\ell_2^c)$. Then for  every $0<s\leq \infty$,  $U\cal{D}_c$ is an   isometric embedding of $\h_s^c(\M)$ into $L_s(\N)$. Let $b \in \h_p^c(\M) \cap \h_q^c(\M)$. Interpolating the operator $U\cal{D}_c$, we have 
\[
\big\| U\cal{D}_c(b)\big\|_{[L_p(\N), L_q(\N)]_\theta} \leq \big\| b \big\|_{[\h_p^c(\M), \h_q^c(\M)]_\theta}.
\]
Since  $\big[(L_p(\N), L_q(\N)\big]_\theta =L_r(\N)$ isometrically, it follows that 
\[
\big\| U\cal{D}_c(b)\big\|_{L_r(\N)} \leq \big\| b \big\|_{[\h_p^c(\M), \h_q^c(\M)]_\theta}.
\]
From the fact that  $U\cal{D}_c$ is an isometry on $\h_r^c(\M)$,  we deduce that
\begin{equation}\label{left}
\big\|b\big\|_{\h_r^c} \leq  \big\| b \big\|_{[\h_p^c(\M), \h_q^c(\M)]_\theta}. 
\end{equation}
From combining \eqref{right} and \eqref{left}, we obtain the desired equivalence.
\end{proof}

When $0<p<1$, we  do not know if  the corresponding  statement  to Theorem~\ref{main-complex} remains valid if  the   interpolation couple $(\h_p^c(\M), \bmo^c(\M))$  is used.   Since reiteration  theorem   is not available for complex interpolations of  quasi-Banach spaces,  in general, this consideration  is independent of Theorem~\ref{main-complex}.
We leave this as an open problem.

\medskip

The same  method of proofs can be applied to conditioned spaces and  spaces of adapted sequences to deduce the following interpolation results from Proposition~\ref{interpolation-conditioned} and Theorem~\ref{interpolation-adapted} respectively.
\begin{proposition}
If $0<\theta <1$ and  $0<p, q<\infty$ then for $1/{r}=(1-\theta)/{p} + {\theta}/q$, the following hold:
\[
\big[ L_p^{\rm cond}(\M; \ell_2^c), L_q^{\rm cond}(\M; \ell_2^c)\big]_{\theta} = L_{r}^{\rm cond}(\M;\ell_2^c)
\]
and
\[
\big[ L_p^{\rm ad}(\M; \ell_2^c), L_q^{\rm ad}(\M; \ell_2^c)\big]_{\theta} = L_{r}^{\rm  ad}(\M;\ell_2^c)
\]
with equivalent  quasi-norms.
\end{proposition}

As  in the case of real interpolation, we may view the second assertion in the proposition as an extension of Musat's result to the quasi-Banach space range. Moreover, 
using \cite[Lemma~1]{Wolff} (which is valid for quasi-Banach spaces), one can adapt  the argument used in the proof of  Corollary~\ref{inclusion-real} to show that the noncommutative analogue of \eqref{H-BMO} holds:
\begin{corollary}\label{inclusion-complex}
Assume that   $0<p<1$ and $0<\theta<\infty$ are such that $1/r=(1-\theta)/p <1$.
Then 
\[
\big[ \H_p^c(\M), \BMO^c(\M) \big]_{\theta}\subset \H_r^c(\M).
\]
\end{corollary}
Our method of proof  only  applies under the assumption that  $r>1$.   We should point out that the reverse  inclusion does not hold even in the classical setting. An example exhibited in \cite[p.~66]{Jason-Jones}
shows that  there exists a probability space $(\Omega, \cal{F},\mathbb{P})$ and an increasing  filtration of $\sigma$-fields $(\cal{F}_n)_{n\geq 1}$ of $\cal{F}$ with $\cal{F}=\sigma(\union_{n\geq 1} \cal{F}_n)$ and  such that   if $0<p<1<r<\infty$,   then for $1/r=(1-\theta)/p$,  $\big[ \H_p(\Omega), \BMO(\Omega) \big]_{\theta} \neq \H_r(\Omega)$.

\medskip

All results stated in this section have row counterparts. However, at the time of this writing, it is unclear   if for $0<p,q<1$,  the  interpolation results for column/row conditioned  Hardy spaces  have counterparts  to the couple of diagonal Hardy spaces $(\h_p^d(\M),\h_q^d(\M))$. 

\section{Applications to martingale inequalities}


In this section, we  present various  martingale inequalities in the  general  framework of noncommutative symmetric spaces that can be derived from  methods we develop in the previous two sections. 

We  will need the following generalization of real interpolation:
\begin{definition}
  An  interpolation space $E$  for a couple  of quasi-Banach spaces $(E_0,E_1)$ is said to be \emph{given by a $K$-method} if there exists a  quasi-Banach function space $\cal{F}$ such that $x \in E$ if and only if $t\mapsto K(x,t ; E_0,E_1) \in \cal{F}$ 
 and there exists constant $C_E>0$  such that
 \[
 C_E^{-1} \big\| t\mapsto K(x,t ; E_0,E_1)\big\|_{\cal{F}} \leq \big\|x\big\|_E \leq C_E \big\| t\mapsto K(x,t ; E_0,E_1)\big\|_{\cal{F}}.
 \]
In this case, we write $E=(E_0, E_1)_{\cal{F};K}$. 
\end{definition}

 \begin{proposition}\label{K-method}
   Let  $0<p<q\leq \infty$. Every interpolation space   $E\in {\rm Int}(L_p,L_q)$ is given by a $K$-method.
  \end{proposition}
  
    For the Banach space range, this fact is known as a result of Brudnyi and Krugliak (see \cite[Theorem~6.3]{KaltonSMS}). An  argument for the   quasi-Banach space range is given    in Dirksen's thesis (\cite{Dirksen-Thesis}). Alternatively, the quasi-Banach space range can be  deduced from the Banach space case  as follows:
assume that $0<p<1$ and $E\in {\rm Int}(L_p,L_q)$. 
Let $E^{(1/p)}$ be the $1/p$-convexification of $E$. That is,
\[
E^{(1/p)}= \Big\{h\in L_0 : |h|^{1/p} \in E\Big\}
\]
equipped with the norm $\| h\|_{E^{(1/p)}}=\| \, |h|^{1/p}\, \|_E^p$. 
 According to   \cite[Corollary~4.6]{Cadilhac1}, $E^{(1/p)} \in {\rm Int}(L_1, L_{q/p})$. Let $\cal{F}$ be a Banach  function space so that $E^{(1/p)} =(L_1, L_{q/p})_{ \cal{F}; K}$. We make the observation from  Homlsted's formula (\cite[Theorem~4.1]{Holm}) that if $a$ is a positive function in $L_p +L_q$ then:
\[
K\big(a, t; L_p, L_q\big) \approx _{p,q}\Big[  K\big( a^p, t^p; L_1, L_{q/p}\big)\Big]^{1/p}.
\]
Let $x \in E$.  We have
\begin{equation*}
\begin{split}
\big\|x\big\|_{E}^{p} &= \big\| \, |x|^{p} \,  \big\|_{E^{(1/p)}}\\
&\approx_E \big\| t\mapsto K(|x|^p, t; L_1, L_{q/p})\big\|_{\cal{F}}\\
&\approx_E \big\| t\mapsto \big[ K(|x|, t^{1/p}; L_p, L_q)\big]^p\big\|_\cal{F}\\
&\approx_E\big\| t\mapsto  K(|x|, t^{1/p}; L_p, L_q)\big\|_{\cal{F}^{(p)}}^p.
\end{split}
\end{equation*}
Let $\cal{Z}=\{ f \in L_0: t\mapsto f(t^{1/p}) \in \cal{F}^{(p)}\}$ with the quasi-norm $\|f\|_{\cal Z}=\|t \mapsto f(t^{1/p})\|_{\cal{F}^{(p)}}$.  We see now that  $E= (L_p, L_q)_{\cal{Z}; K}$.  \qed 
 
 \medskip

 Below, we will also use the $J$-method  version of the interpolation associated with function space which we now briefly describe: for $x \in E_0 +E_1$, we recall that by a representation of $x$ with respect to the couple $(E_0, E_1)$, we mean a measurable function $u: (0,\infty) \to   E_0 \cap E_1$ satisfying
 \[
 x =\int_0^\infty  u(t)\ \frac{dt}{t}
 \]
where the convergence is taken in $E_0 +E_1$. Recall that for $y \in E_0\cap E_1$ and $t>0$,
\[
J(y, t ; E_0,E_1) = \max\{\|y\|_{E_0}; t \|y\|_{E_1}\}.
\]
 For  a  given function space $\cal{F}$, we define the quasi-norm
 \[
 \|x\|_{(E_0,E_1)_{\cal{F};J}}:=\inf\big\{\|t \mapsto J(u(t), t ;E_0,E_1)\|_{\cal{F}}\big\}
 \]
where the infimum is taken over all representation $u(\cdot)$ of $x$ with respect  to the couple $(E_0, E_1)$. The  interpolation space $(E_0,E_1)_{\cal{F};J}$ is defined as the collection of all $x \in E_0 + E_1$ for which $\|x\|_{(E_0,E_1)_{\cal{F};J}}<\infty$.
We refer to \cite{KPS}  for more on this interpolation method.

\smallskip
 
 From                                                                                                                                                                                                                           the fact that every interpolation space  of the couple $(L_p, L_q)$ is given by a $K$-method, the following  can  be easily deduced from our results on $K$-functionals from the previous section:
\begin{proposition}\label{int-general}
Let $\nu$ be a positive  integer with $\nu\geq 2$. Assume that $2/(\nu +1) <p\leq 2/\nu$ and $ p<q\leq 2/(\nu-1)$. If $E \in {\rm Int}(L_p, L_q)$ with $E=(L_p,L_q)_{\cal{F};K}$ for a quasi Banach function space  $\cal{F}$
 then  the following hold:
\[\h_E^c(\M) =( \h_p^c(\M), \h_q^c(\M)\big)_{\cal{F};K},
\]
\[
E^{\rm cond}(\M;\ell_2^c) =\big( L_p^{\rm cond}(\M;\ell_2^c), L_q^{\rm cond}(\M;\ell_2^c)\big)_{\cal{F};K},
\]
and 
\[
E^{\rm ad}(\M;\ell_2^c) =\big( L_p^{\rm ad}(\M;\ell_2^c), L_q^{\rm ad}(\M;\ell_2^c)\big)_{\cal{F};K}.
\]
\end{proposition}

The first assertion in the preceding  observation motivates the following more general question:  assume that $E_0$ and $E_1$ are symmetric quasi-Banach function spaces and $E\in {\rm Int}(E_0, E_1)$, does  it follow that $\h_E^c(\M) \in {\rm Int}\big(\h_{E_0}^c(\M), \h_{E_1}^c(\M)\big)$?

Note that if for $j\in\{0,1\}$, $E_j \in {\rm Int}(L_{p_j}, L_{q_j})$ for some  $1<p_j,q_j<\infty$, then one can deduce from Junge's representation  and interpolation that $\h_{E_j}^c(\M)$ embeds complementably  into $E_j(\M\overline{\otimes} B(\ell_2(\mathbb{N}^2)))$. Thus, in this special case, the answer to the  above question  is clearly positive. Even for the particular case where $E_0=L_p$ and $E_1=L_q$, we do not know if  the assumptions on $p$ and $q$ in Proposition~\ref{int-general} can be removed when $0<p<q\leq 1$. This of course is closely related to asking whether  the statement  about $K$-functionals in Proposition~\ref{Key-lemma}  is valid  for any $0<p<q \leq 1$.

 Next,  we make the   observation  that  a well-known result  for convex functions extends to the general setting of $p$-convex functions for $0<p<1$.
\begin{proposition}\label{Phi-interpol}
Let $0<p<q<\infty$, and  $E \in {\rm Int}(L_p, L_q)$. There exists a constant $c_E$ such that the following holds: if $f$, $g$ are functions such that  $\|g\|_{L_\Phi} \leq  \|f\|_{L_\Phi} $ for every function   $\Phi$  that  is $p$-convex and $q$-concave, and if $f \in E$,  then $g\in E$ with 
\[
\|g\|_E \leq c_E \|f\|_E.
\]
\end{proposition}
\begin{proof}
Let $f$ and $g$ as in the statement of the proposition.
Using  the Orlicz space description  of $L_p +tL_q$ in the previous section,   we have  from the assumption that
\[
K(g ,t ; L_p, L_q) \lesssim_{p,q} K(f,t ; L_p, L_q), \quad t>0.
\]
According to Proposition~\ref{K-method},  the interpolation space $E$ is given by a $K$-method. Fix a function space $\cal{F}$ so that for every $h \in E$, 
 \[
 \|h\|_E \approx_{E} \| t \mapsto K(h,t; L_p, L_q)\|_{\cal{F}}.
 \] 
 We may now deduce from  the  inequality on $K$-functionals  that
 \begin{align*}
 \|g\|_E &\approx_{E} \| t \mapsto K(g,t; L_p, L_q)\|_{\cal{F}}\\
 &\lesssim_E \| t \mapsto K(f,t; L_p, L_q)\|_{\cal{F}}\\
 &\approx_{E} \| f\|_E.
 \end{align*}
 This verifies the desired conclusion.
\end{proof}

Our first result in this section is a comparison  between conditioned  column  space and  column space for  a  class of symmetric spaces.

\begin{theorem}\label{reverse-cond} 
Let  $0<p<q<2$ and $F\in  {\rm Int}(L_p, L_q)$. There exists a constant $C_F$  such that   for any  $x \in F^{\rm cond}(\M;\ell_2^c)$, the following holds:
\[
\big\| x \big\|_{F(\M;\ell_2^c)}\leq C_F \big\| x \big\|_{F^{\rm cond}(\M;\ell_2^c)}.
\]
Similarly, if $\Phi$ is an Orlicz function that is $p$-convex and $q$-concave for $0<p<q<2$ then there exists a constant $C_{p,q}$ so that for any sequence $x=(x_k)_{k\geq 1}$ with $\sigma_c(x) \in L_\Phi(\M)$,
\[
\T\big[\Phi\big( \cal{S}_c(x) \big)\big] \leq C_{p,q}\T\big[\Phi\big(\sigma_c(x)\big)\big]
\]
\end{theorem}
\begin{proof}
Let $\Phi$ be an Orlicz function that is $p$-convex and $q$-concave.
According to Theorem~\ref{main-algebraic}, if $x=(x_k)_{k\geq 1} $ in $\mathfrak{F}$, the column vector   $\overline{x}=\sum_{k\geq 1} x_k \otimes e_{k,1}$ admits a factorization  $\overline{x}=\a\, . \, \b$ where $\a$ is a strictly lower triangular matrix taking values in $L_2(\M)$ and $\b \in L_{\Theta}^{\rm ad}(\M;\ell_2^c)$ where $\Theta$ is the  Orlicz function satisfying $L_\Phi=L_2 \odot L_\Theta$ (as described in Theorem~\ref{main-algebraic})  and 
\[
\big\| \a\big\|_{L_2(\M \overline{\otimes} B(\ell_2(\mathbb N)))} \ . \ \big\|\b\big\|_{L_\Theta(\M;\ell_2^c)}  \leq C_{p,q} \big\| x \big\|_{L_\Phi^{\rm cond}(\M;\ell_2^c)}.
\]
From  the factorization of $L_\Phi$, it follows that 
\[
\big\| x \big\|_{L_\Phi(\M;\ell_2^c)} =
\big\| x \big\|_{L_\Phi(\M \overline{\otimes} B(\ell_2(\mathbb{N})))} \leq  \big\| \a\big\|_{L_2(\M \overline{\otimes} B(\ell_2(\mathbb N)))} \ . \ \big\|\b\big\|_{L_\Theta(\M;\ell_2^c)}.
\]
Combining the two inequalities leads to
\[
\big\| \cal{S}_c(x) \big\|_{L_\Phi(\M)} \leq  C_{p,q} 
\big\| \sigma_c(x) \big\|_{L_\Phi(\M)}. 
\]
By density, we may restate this as there is a map $J: L_\Phi^{\rm cond}(\M;\ell_2^c) \to L_\Phi(\M;\ell_2^c)$ with $J(x)=x $ for $x \in \mathfrak{F}$. For simplicity, we write $J(y)=y$ for arbitrary $y \in L_\Phi^{\rm cond}(\M;\ell_2^c)$. 

Now let $\xi$ be an element of $ F^{\rm cond}(\M; \ell_2^c)$. Then, by definition, $U(\xi) \in F(\M \overline{\otimes} B(\ell_2(\mathbb{N}^2)))$  and from boundedness of $J$ implies
\[
\big\| \mu( \xi) \big\|_{L_\Phi} \leq C_{p,q} \big\| \mu( U(\xi)) \big\|_{L_\Phi}
\]
where the generalized singular value  on the left hand side is taken with  respect  to $\M \overline{\otimes} B(\ell_2(\mathbb{N}))$ and the one the right hand side is taken with respect to $\M \overline{\otimes} B(\ell_2(\mathbb{N}^2))$. It follows from Proposition~\ref{Phi-interpol} that there exists a constant $c_F$ such that
\[
\big\| \mu( \xi) \big\|_F \leq C_{F} \big\| \mu( U(\xi)) \big\|_F.
\]
This is equivalent to 
\[
\big\| \xi\big\|_{F(\M;\ell_2^c)} \leq C_F \big\| \xi\big\|_{F^{\rm cond}(\M;\ell_2^c)}
\]
which is the desired conclusion.

\smallskip

For the $\Phi$-moment version, it suffices to repeat the above argument but using Remark~\ref{moment-algebraic} in place of Theorem~\ref{main-algebraic}.
\end{proof}

As an  immediate consequence of Theorem~\ref{reverse-cond}, we have the   following extension of the   reverse dual Doob inequality  proved in \cite[Theorem~7.1]{JX} for noncommutative $L_p$  spaces ($0<p<1$) to the general case of  noncommutative symmetric quasi-Banach  spaces.
 \begin{corollary} Let $E$ be a  symmetric quasi-Banach  function space with $E\in {\rm Int}(L_p,L_q)$ for $0<p<q<1$. There exists a constant $C_E$ so that for any sequence of positive operators $(a_k)$  in $\mathfrak{F}$, the following holds:
 \[
 \Big\| \sum_{k\geq 1} a_k \Big\|_{E(\M)} \leq C_E \Big\| \sum_{k\geq 1} \E_k(a_k)\Big\|_{E(\M)}.
 \] 
 
 Similarly,  if $\Phi$  is an Orlicz space that is $p$-convex and $q$-concave for $0<p<q<1$,  then there exists a constant $C_{p,q}$ so that for any  sequence of positive operators $(a_k)$  in $\mathfrak{F}$, the following holds:
 \[
\T\big[ \Phi\big(\sum_{k\geq 1} a_k\big) \big]  \leq C_{p,q} \T\big[\Phi\big( \sum_{k\geq 1} \E_k(a_k)\big)\big]
 \]
 \end{corollary}

\begin{proof}
Let $E\in {\rm Int}(L_p, L_q)$ with $0<p<q<1$.
Let $(a_k)_{k\geq 1}$ be a sequence of positive operators in $\mathfrak{F}$.   If $E^{(2)}$ is the $2$-convexification of $E$, then $E^{(2)}\in {\rm Int}(L_{2p}, L_{2q})$.   The conclusion follows immediately from  the first inequality in Theorem~\ref{reverse-cond} using $F=E^{(2)}$ and $(x_k)=(a_k^{1/2})$. Similar argument applies to the $\Phi$-moment case using  
the second inequality in Theorem~\ref{reverse-cond} to the Orlicz function $t\mapsto \Phi(t^2)$.
\end{proof}


The next result extends   \cite[Theorem~4.11]{Jiao-Ran-Wu-Zhou} to the case of  noncommutative   martingale Hardy spaces associated with  symmetric function spaces and moment inequalities.
\begin{theorem} Let $0<p<q<2$.
If  $F\in {\rm Int}(L_p, L_q)$  then there exists a constant $C_F$ such that for every $x \in \h_F^c(\M)$, the following two inequalities hold:
\[
\big\|x\big\|_{\H_F^c(\M)} \leq C_F \big\|x\big\|_{\h_F^c(\M)}
\]
and
\[
\big\|x\big\|_{F(\M)} \leq C_F \big\|x\big\|_{\h_F^c(\M)}.
\]
Similarly, if $\Phi$ is $p$-convex and $q$-concave for $0<p<q<2$ then  there exists a constant $C_{p,q}$ so that for every $x\in \h_\Phi^c(\M)$, we have
\[
\max\Big\{ \T\big[ \Phi\big(S_c(x)\big)\big] ;  \T\big[ \Phi\big(|x|\big)\big]  \Big\} \leq C_{p,q} 
\T\big[ \Phi\big(s_c(x)\big)\big].
\]
\end{theorem}
\begin{proof}  $\bullet$ For the first inequality, let $y$ be a  martingale in $\mathfrak{F}(M)$.  Repeating the argument in the proof of Theorem~\ref{reverse-cond} with the martingale difference  sequence $(dy_k)$, we have  for every $p$-convex and $q$-concave Orlicz function $\Phi$:
\[
\big\| y\big\|_{\H_\Phi^c(\M)} \leq C_{p,q} \big\|y\big\|_{\h_\Phi^c(\M)}.
\]
By density, the above inequality shows that there exists a bounded linear map $I: \h_\Phi^c(\M) \to \H_\Phi^c(\M)$
with $I(y)=y$ for every  $y \in \mathfrak{F}(M)$. For simplicity, for a given $x \in \h_\Phi^c(\M)$, we will denote  $I(x)$  by $x$. That is,  for $x\in \h_\Phi^c(\M)$ we may state:
\[
\big\| \mu(S_c(x)) \big\|_{L_\Phi} \leq C_{p,q} \big\| \mu(| U{\cal D}_c(x)|)\big\|_{L_\Phi}
\]
where the generalized singular numbers are computed in the appropriate von Neumann algebras.
By Proposition~\ref{Phi-interpol}, there exists a constant $C_F$ so that
whenever  $ x \in \h_F^c(\M)$, we have  $\mu(| U{\cal D}_c(x)|)\in F$ and  
\[
\big\| \mu(S_c(x)) \big\|_{F} \leq C_F \big\| \mu(| U{\cal D}_c(x)|)\big\|_{F} .
\]
This is equivalent to the first inequality:
\[
\big\| x\big\|_{\H_F^c(\M)} \leq C_F \big\|x \big\|_{\h_F^c(\M)}.
\]

$\bullet$ In view of the  argument above, it suffices to verify the second  inequality  for the case of Orlicz function spaces. For this special case,  our proof below is modeled after the argument used  in  \cite[Corollary~3.14]{Chen-Ran-Xu} for the case of $L_p$-spaces. Let  $\Phi$   be a $p$-convex and $q$-concave Orlicz function and $x$ be a  martingale in $\mathfrak{F}(M)$. 
By approximation, we assume that  for every $n\geq 1$, $s_{c,n}(x)$  is invertible with bounded inverse. As above, we denote $s_{c,n}(x)$ by $s_n$ and we take $s_0=0$. 

Let $\l>0$. We write $x=\sum_{l\geq 1} a_lb_l$ where for every $l\geq 1$, we set:
 \[
 a_l=\sum_{n\geq l} dx_n \Psi(\l s_n)(\Psi(\l s_l)^{-1} -\Psi(\l 
 s_{l-1})^{-1})^{1/2}
 \ \ 
  \text{and}\ \  
  b_l=(\Psi(\l s_l)^{-1} -\Psi(\l s_{l-1})^{-1})^{1/2}. \]
   Using the factorization $L_\Phi=L_2 \odot L_\Theta$, we may deduce that:
\begin{align*}
\big\|x\big\|_{L_\Phi(\M)} &= \Big\| \sum_{l\geq 1} a_l b_l\Big\|_{L_\Phi(\M)}\\
&\leq \Big\| \Big( \sum_{l\geq 1} a_l a_l^*\Big)^{1/2} \Big\|_2 . \Big\|\Big( \sum_{l\geq 1} b_l^* b_l\Big)^{1/2}\Big\|_{L_\Theta(\M)}\\
&=\Big( \sum_{l\geq 1}\big\|a_l\big\|_2^2\Big)^{1/2} \big\| \Psi(\l s)^{-1/2}\big\|_{L_\Theta(\M)}.
\end{align*}
Following the same reasoning as  in the proof of Theorem~\ref{main-algebraic}, there exists a constant $C_{p,q}$ and a corresponding choice of $\l$ so that
\[
\Big( \sum_{l\geq 1}\big\|a_l\big\|_2^2\Big)^{1/2} \leq  C_{p,q} \big\|s_c(x)\big\|_{L_\Phi(\M)}
\]
and 
\[
\big\| \Psi(\l s)^{-1/2}\big\|_{L_\Theta(\M)} \leq 1.
\]
This  yields  that the inequality in the statement is verified for  Orlicz function spaces. More precisely,
\begin{equation}\label{O}
\big\|x\big\|_{L_\Phi(\M)}  \leq C_{p,q} \big\|s_c(x)\big\|_{L_\Phi(\M)}.
\end{equation}
The case of moment inequalities follows directly from the moment part of Theorem~\ref{reverse-cond}  and from using $\lambda=1$ in the proof above. 
\end{proof}


\medskip

We  now proceed with  further application of Proposition~\ref{factorization-adapted}. Below, we show that  noncommutative  Davis decompositions can be easily deduced from  factorizations of adapted sequences. We refer to  \cite{Junge-Perrin}, \cite{Perrin},  and \cite{Ran-Wu-Xu}  for  various forms of  noncommutative   Davis decompositions.  We refer to \cite{Ran-Wu-Zhou} for formal definition of  the noncommutative vector-valued space $L_\Phi(\M;\ell_1^c)$ used below.

\begin{theorem}\label{Davis-Orlicz}  Assume that $0<p<q<2$ and $\Phi$ is an Orlicz  function that is $p$-convex and $q$-concave.  Given an adapted sequence $\xi=(\xi_n)_{n\geq 1}$  in $L_\Phi(\M;\ell_2^c)$, there exist two adapted sequences $y=(y_n)_{n\geq 1}$ and $z=(z_n)_{n\geq 1}$  such that:
\begin{enumerate}[{\rm(i)}]
\item $\xi=y +z$;
 \item $y \in L_\Phi(\M;\ell_1^c)$ with $\big\| y \big\|_{L_\Phi(\M;\ell_1^c)} \lesssim_{p,q} \big\| \xi \|_{L_\Phi(\M;\ell_2^c)}$;
\item $z=\lambda a$ where $a$ is an algebraic $L_\Phi^{c,\rm cond}$-atom and 
$\lambda\lesssim_{p,q} \big\| \xi \|_{L_\Phi(\M;\ell_2^c)}$.
\end{enumerate}
In particular,
\[
\big\| y \big\|_{L_\Phi(\M;\ell_1^c)}  + \big\| z \big\|_{L_\Phi^{c,\rm cond}(\M;\ell_2^c)} \lesssim_{p,q} \big\| \xi \|_{L_\Phi(\M;\ell_2^c)}.
\]
\end{theorem}
\begin{proof}
Let $\Theta$ be an Orlicz  function such that $L_\Phi=L_2 \odot L_\Theta$. Consider the factorization $\xi= \a \,  .\, \b$ where $\a$ is a lower triangular matrix in $L_2(\M \overline{\otimes} B(\ell_2(\mathbb N)))$ and $\b \in L_\Theta^{\rm ad}(\M;\ell_2^c)$
according  to Proposition~\ref{factorization-adapted}. Define the strictly lower triangular matrix $\a^-$  by setting $\a^-_{n,j}=\a_{n,j}$ for $1\leq j <n$ and  the diagonal matrix $d= \sum_{n\geq 1} \a_{n,n} \otimes e_{n,n}$. Clearly,
$\a= \a^- +d$.  Set:
\[
y= d\, . \, \b \quad\text{and}\quad z= \a^-\, . \, \b.
\]
Then $\xi=y+z$ and from Proposition~\ref{factorization-adapted}(ii), $y$ and $z$ are adapted. Since $\a^-$ is strictly lower triangular and $\b$ is adapted, $z$ is a scalar multiple of an algebraic $L_\Phi^{c, \rm cond}$-atom. The verification of the  norm estimates are straightforward.
\end{proof}

We now turn our attention to Davis type inequalities   involving   other classes of symmetric spaces of measurable operators using  interpolation results from the previous section.  The next result deals with the case of Lorentz spaces.
\begin{proposition}\label{simultaneous} Let   $\xi=(\xi_n)_{n\geq 1}$ be an adapted sequence that belongs to  $L_{2/3}(\M;\ell_2^c)\cap L_2(\M;\ell_2^c)$. Then there exist   two  adapted sequences $y=(y_n)_{n\geq 1}$ and $z=(z_n)_{n\geq 1}$  in $L_{2/3}(\M)\cap L_2( \M)$ such that:
\begin{enumerate}[{\rm(i)}]
\item $\xi=y +z$;
\item  for every $2/3<p<2$ and  $0<q\leq \infty$,  the following holds:
\[
\big\|y\big\|_{L_{p,q}(\M \overline{\otimes}\ell_\infty)} + \big\| z\big\|_{L_{p,q}^{\rm cond}(\M;\ell_2^c)} \lesssim_{p,q}\big\|\xi\big\|_{L_{p,q}(\M;\ell_2^c)}.
\]
\end{enumerate}
\end{proposition}
\begin{proof}
The key ingredients for the proof are the decomposition from \cite[Theorem~3.1]{Ran-Wu-Xu} and the two interpolations results from Proposition~\ref{interpolation-conditioned} and Theorem~\ref{interpolation-adapted}.

Fix 
$\xi \in L_{2/3}^{\rm ad}(\M;\ell_2^c)  \cap L_2^{\rm ad}(\M;\ell_2^c)$. 
 Choose a discrete representation of $\xi$ with respect to the couple $(L_{2/3}^{\rm ad}(\M;\ell_2^c),  L_2^{\rm ad}(\M;\ell_2^c))$,
 \[
\xi= \sum_{\nu \in \Z} \xi^{(\nu)}
\]
with $\xi^{(\nu)}  \in  L_{2/3}^{\rm ad}(\M;\ell_2^c) \cap  L_2^{\rm ad}(\M;\ell_2^c)$ for every $\nu \in  \Z$, the series is convergent in $L_{2/3}^{\rm ad}(\M;\ell_2^c) +  L_2^{\rm ad}(\M;\ell_2^c)$,  and  
so that 
\[
J\big( \xi^{(\nu)}, 2^\nu, L_{2/3}^{\rm ad}(\M;\ell_2^c); L_{2}^{\rm ad}(\M;\ell_2^c)\big) \leq 4 K\big( \xi, 2^\nu, L_{2/3}^{\rm ad}(\M;\ell_2^c); L_2^{\rm ad}(\M;\ell_2^c)\big). 
\]
We refer to \cite[Lemma~3.3.2]{BL} for the existence of a representation  satisfying the properties discribed above.

For each $\nu\in \mathbb{Z}$,  we apply \cite[Theorem~3.1]{Ran-Wu-Xu} to the adapted sequence $\xi^{(\nu)}$: there exists two adapted sequences $y^{(\nu)}$ and $z^{(\nu)}$  such that $\xi^{(\nu)}=y^{(\nu)} + z^{(\nu)}$,
\[
J\big( y^{(\nu)}, 2^\nu, L_{2/3}(\M\overline{\otimes} \ell_\infty;  L_{2}(\M\overline{\otimes} \ell_\infty))\big)\lesssim J\big( \xi^{(\nu)}, 2^\nu, L_{2/3}^{\rm ad}(\M;\ell_2^c); L_2^{\rm ad}(\M;\ell_2^c)\big), 
\]
and
\[
J\big( z^{(\nu)}, 2^\nu, L_{2/3}^{\rm cond}(\M;\ell_2^c); L_2^{\rm cond}(\M;\ell_2^c)\big)
\lesssim  J\big( \xi^{(\nu)}, 2^\nu, L_{2/3}^{\rm ad}(\M;\ell_2^c); L_2^{\rm ad}(\M;\ell_2^c)\big).  
\]
It follows that for $0<\theta<1$ and $0<q\leq \infty$, we have:
\begin{equation}\label{J-diag}
\big\| (J(y^{(\nu)}, 2^\nu ))_\nu\big\|_{\lambda^{\theta,q}}\lesssim_{\theta,q}  \big\| (K(\xi, 2^\nu ))_\nu\big\|_{\lambda^{\theta,q}}
\end{equation}
and
\begin{equation}\label{J-cond}
\big\| (J(z^{(\nu)}, 2^\nu ))_\nu\big\|_{\lambda^{\theta,q}}\lesssim_{\theta,q} \big\| (K(\xi, 2^\nu ))_\nu \big\|_{\lambda^{\theta,q}}
\end{equation}
 where for a given sequence $(a_\nu)_{\nu}$ of scalars, we use the quasi-norm 
\[
\big\| (a_\nu)_{\nu}\big\|_{\lambda^{\theta,q}}:= \Big( \sum_{\nu  \in  \Z} \big(2^{-\nu\theta} |a_\nu|\big)^q\Big)^{1/q}.
\]
Let  $y= \sum_{\nu \in \Z} y^{(\nu)}$  and $z=\sum_{\nu \in \Z} z^{(\nu)}$.  It is clear $y$ and $z$ are adapted sequences and $\xi=y+z$.  
For any given $2/3<p<2$ and $0<q\leq \infty$,  choose $\theta=3/2-1/p$.
It  clearly follows from \eqref{J-diag}  and Theorem~\ref{interpolation-adapted} that:
\[
\big\|y \big\|_{L_{p,q}(\M\overline{\otimes} \ell_\infty)} \lesssim_{p,q}  \|\xi\|_{L_{p,q}^{\rm ad}(\M;\ell_2^c)}.
\]
Similarly,  we may deduce from \eqref{J-cond} and Proposition~\ref{interpolation-conditioned} that
\[
\big\|z \big\|_{L_{p,q}^{\rm cond}(\M ; \ell_2^c)} \lesssim_{p,q}  \|\xi\|_{L_{p,q}^{\rm ad}(\M;\ell_2^c}.
\]
The desired inequality follows from combining the last two inequalities.
\end{proof}
We isolate the following  important example. Motivated by  the  noncommutative Khintchine inequality for weak-$L_1$ spaces (\cite{Cadilhac2}),  we state below   a weak-$L_1$-version of the Davis-decomposition for adapted sequences. This appears to be new even for the classical setting.
\begin{example}
There exists a constant $C>0$  so that  for every adapted sequence $\xi \in L_{1,\infty}(\M;\ell_2^c)$, there exist  two adapted sequences  $y$ and $z$ such that:
\begin{enumerate}[{\rm (i)}]
\item $\xi=x +z$;
\item $\big\| y\big\|_{L_{1,\infty}(\M \overline{\otimes} \ell_\infty)} + \big\| z \big\|_{L_{1,\infty}^{\rm cond}(\M ; \ell_2^c)} \leq  C \big\| \xi \big\|_{L_{1,\infty}(\M;\ell_2^c)}$.
\end{enumerate}
\end{example}
It is worth pointing out   that  the preceding example  allows us to deduce the noncommutative  weak-type $(1,1)$ version of the Burkholder/Rosenthal 
(\cite[Theorem~3.1]{Ran21}) from  the simpler  weak-type inequality involving  square functions given in  \cite[Theorem~2.1]{Ran16}.

\medskip

The idea used in the proof of Proposition~\ref{simultaneous} can be extended for the case of general symmetric spaces for the Banach space range. More precisely, we  have the following result:

\begin{proposition}\label{Davis-symmetric}
Let $E$ be a Banach function space. 
If  $E\in {\rm Int}(L_1,L_q)$ for $1<q<2$, then 
\[
E^{\rm ad}(\M;\ell_2^c) =  E(\oplus_{n=1}^\infty \M_n) +  E^{\rm cond, ad}(\M;\ell_2^c)
\]
where  $ E^{\rm cond, ad}(\M;\ell_2^c)$ denotes is the subspace of $E^{\rm cond}(\M;\ell_2^c)$ consisting of adapted sequences.
\end{proposition}
\begin{proof}[Sketch of the proof]
It is easy to see that under the assumption, $E(\oplus_{n=1}^\infty \M_n)\subseteq E^{\rm ad}(\M;\ell_2^c) $.  On the other hand, we have from Theorem~\ref{reverse-cond}  that
$E^{\rm cond, ad}(\M;\ell_2^c) \subseteq E^{\rm ad}(\M;\ell_2^c)$. Thus,  we only need to verify one inequality.

The proof rests upon few facts. The interpolation  space  $E$ is given by a  $K$-method, the simultaneous nature of Proposition~\ref{simultaneous} above, and Proposition~\ref{int-general}.
 
Since  $E$ is   given by  a $K$-method, we may fix a   Banach function space $\cal{F}$  such that  for any given semifinite von Neumann algebra $\cal N$,
\[
C_E^{-1}\big\|a \big\|_{\cal{F} ; K} \leq \big\|a \big\|_{E(\cal N)}  \leq \big\|a \big\|_{\cal{F}; K}\ , \quad a\in E(\cal N).
\]
Under the assumption $1<q<2$ and Proposition~\ref{int-general}, we may also state that for every sequence $\xi \in E^{\rm ad}(\M;\ell_2^c)$,
\[
\big\| \xi\big\|_{\big(L_1^{\rm ad}(\M;\ell_2^c); L_q^{\rm ad}(\M;\ell_2^c)\big)_{\cal{F};K}} \approx_E
\big\|\xi\big\|_{E^{\rm ad}(\M;\ell_2^c)}.
\]
Similarly,  for every    $\zeta  \in E^{\rm  cond }(\M;\ell_2^c)$,
\[
\big\| \zeta\big\|_{\big(L_1^{\rm cond}(\M;\ell_2^c); L_q^{\rm cond}(\M;\ell_2^c)\big)_{\cal{F};K}} \approx_{E}
\big\|\zeta \big\|_{E^{\rm cond}(\M;\ell_2^c)}.
\]
We now outline the argument. Fix $\xi \in  L_1^{\rm ad}(\M;\ell_2^c) \cap   L_q^{\rm ad}(\M;\ell_2^c)\big)$.   Repeating  the argument in the proof of Proposition~\ref{simultaneous} (taking into account the fact that the decomposition in Proposition~\ref{simultaneous} works simultaneously), we obtain a decomposition $\xi=y +z$ where $y=\sum_{\nu \in \Z} y^{(\nu)}$ and $z=\sum_{\nu \in \Z} z^{(\nu)}$
are representations with respect to the couple $(L_1(\M\overline{\otimes}\ell_\infty),L_q(\M\overline{\otimes}\ell_\infty))$ and $(L_1^{\rm cond}(\M; \ell_2^c),L_q^{\rm cond}(\M;\ell_2^c))$ respectively  and further  satisfy that for every $\nu \in \Z$,
\[
J\big( y^{(\nu)}, 2^\nu, L_1(\M\overline{\otimes}\ell_\infty); L_q(\M \overline{\otimes} \ell_\infty)\big)  \lesssim_{q} K\big( \xi, 2^\nu, L_1^{\rm ad}(\M;\ell_2^c); L_q^{\rm ad}(\M;\ell_2^c)\big) 
\]
and
\[
J\big( z^{(\nu)}, 2^\nu, L_1^{\rm cond}(\M;\ell_2^c); L_q^{\rm cond}(\M;\ell_2^c)\big)
\lesssim_{q}  K\big( \xi, 2^\nu, L_1^{\rm ad}(\M;\ell_2^c); L_q^{\rm ad}(\M;\ell_2^c)\big).  
\]
 Consider the following  functions defined on the semi-axis $(0,\infty)$:
 \[
 f(t) = J(y^{(\nu)}, 2^\nu)\quad \text{for} \ t \in [2^\nu, 2^{\nu+1}),
 \]
\[
 g(t) = J(z^{(\nu)}, 2^\nu)\quad \text{for} \ t \in [2^\nu, 2^{\nu+1}),
\]
and
\[
h(t)=K(\xi, 2^\nu)\quad \text{for} \ t \in [2^\nu, 2^{\nu+1}).
\] 
It follows that  for every $t>0$, 
\[
\max\big\{ f(t); g(t) \big\} \lesssim_{q}  h(t).
\]
Taking the norms on the function space $\cal{F}$, we have
\[
\big\|f\big\|_{\cal{F}} + \big\|g\big\|_{\cal{F}} \lesssim_q \big\|h\big\|_{\cal{F}}.
\]
From the definitions of the three functions, we further get that:
\begin{equation}\label{J-K}
\|y\|_{(L_1(\M\overline{\otimes}\ell_\infty), L_q(\M\overline{\otimes}\ell_\infty)_{\cal{F};J}}  +
\|z\|_{(L_1^{\rm cond}(\M;\ell_2^c); L_q^{\rm cond}(\M;\ell_2^c))_{\cal{F};J}}
\lesssim_q
\|\xi\|_{(L_1^{\rm ad}(\M;\ell_2^c); L_q^{\rm ad}(\M;\ell_2^c))_{\cal{F};K}}.
\end{equation}
From the equivalence of the $J$-methods and $K$-methods relative the function space $\cal{F}$ (see for instance, \cite[Theorem~2.9]{KPS}) and the equivalence of norms stated at the beginning of the proof, we may conclude that:
\[
\|y\|_{E(\M\overline{\otimes} \ell_\infty)} + \|z\|_{E^{\rm cond}(\M;\ell_2^c)} \lesssim_E
\|\xi\|_{E^{\rm ad}(\M;\ell_2^c)}
\] 
which is the desired inequality.
\end{proof}

\begin{remark} The argument in the proof of Theorem~\ref{Davis-symmetric} can be carried out for the larger class $E\in {\rm Int}(L_p, L_q)$  with $2/3<p<q<2$ to get that every $\xi \in L_{p}^{\rm ad}(\M;\ell_2^c) \cap L_{q}^{\rm ad}(\M;\ell_2^c)$ admits a decomposition 
into two adapted sequences  $y$ and $z$ satisfying:
\[
\|y\|_{(L_p(\M\overline{\otimes}\ell_\infty), L_q(\M\overline{\otimes}\ell_\infty)_{\cal{F};J}}  +
\|z\|_{(L_p^{\rm cond}(\M;\ell_2^c); L_q^{\rm cond}(\M;\ell_2^c))_{\cal{F};J}}
\lesssim_{E}
\|\xi\|_{ E^{\rm ad}(\M;\ell_2^c)}.
\] 
However, we do not know if the equivalence of  the $J$-method and  the $K$-method relative to the function space $\cal{F}$ is valid for the case of quasi-Banach couples.
\end{remark}

 As an  immediate application  of Proposition~\ref{Davis-symmetric},  we deduce the next result   which partially answers a problem  from \cite[Remark~3.11]{Ran-Wu-Xu}. 
\begin{corollary}\label{davis}
Let $E$ be a Banach function space.
If $E \in {\rm Int}(L_1, L_q)$  for $1<q<2$,  then  
\[
\H_E^c(\M) = \h_E^d(\M) + \h_E^c(\M).
\]
\end{corollary}
\begin{proof} We only need  to verify  the inclusion $\H_E^c(\M) \subseteq \h_E^d(\M) + \h_E^d(\M)$. That is to show the existence of a constant $\a_E$  such that for every $x \in \H_E^c(\M)$, the following holds:
\[
\inf\left\{ \big\|x^d\big\|_{\h_E^d} + \big\|x^c\big\|_{\h_E^c} \right\} \leq \a_E \big\|x \big\|_{\H_E^c}
\]
where the infimum is taken over all $x^d \in \h_E^d(\M)$ and $x^c \in \h_E^c(\M)$ such that $x=x^d + x^c$.

Let $\xi=(dx_n)_{n\geq 1} \in E^{\rm ad}(\M;\ell_2^c)$.
 It is enough to take
martingales $x^c$ and $x^d$ with  for every $n\geq 1$,
\[
dx_n^c= z_n - \E_{n-1}(z_n) \quad  \text{and}  \quad dx_n^d= y_n - \E_{n-1}(y_n), 
\]
where $y$ and $z$ are the adapted sequences from Theorem~\ref{Davis-symmetric}. Clearly, $x=x^d +x^c$. 
Since the map $(a_n)_{n\geq 1}\mapsto (\E_{n-1}(a_n))_{n\geq 1} $ is  a contraction  in $L_p(\oplus_{n=1}^\infty \M_n)$ for every $1\leq p<\infty$,  it follows by interpolation that it is also bounded in $E(\oplus_{n=1}^\infty \M_n)$. In particular,
\[
\|x^d\|_{\h_E^d} \lesssim_E \|y\|_{E(\oplus_{n=1}^\infty \M_n)}.
\]
On the other hand,  since for every $n\geq 1$, $\E_{n-1}|dx_n^c|^2 \leq \E_{n-1}|z_n|^2$, we immediately get  that 
\[
\|x^c\|_{\h_E^c} \leq \|z\|_{E^{\rm cond}(\M;\ell_2^c)}. 
\]
 Combining these two inequalities, we  arrive at
 \[
 \|x^d\|_{\h_E^d} + \|x^c\|_{\h_E^c}  \lesssim_E \|\xi\|_{E^{\rm ad}(\M;\ell_2^c)} =\|x\|_{\H_E^c}.
 \]
 The proof is complete.
 \end{proof}

We recall that  the conclusion of Corollary~\ref{davis} also applies to interpolation space   $E \in {\rm Int}(L_p, L_2)$ for $1<p<2$ (see \cite[Theorem~3.9]{Ran-Wu-Xu}).
We suspect that the preceding  corollary is valid  for any Banach function space in ${\rm Int}(L_1, L_2)$ but our method  is restricted to  $1<q<2$ (see also \cite[Problem~4.2]{RW} for a related question). 

As examples of spaces that  are not covered by previously known results, we consider the general Lorentz space $\Lambda_{1, w}$ where $w$ is a positive  decreasing function on  $(0,\infty)$ with $\int_0^\infty w(t)\ dt=\infty$. The Lorentz space $\Lambda_{1,w}$ is  the linear space consisting of all  $ f\in  L_0$  such  that
\[
\big\|f\big\|_{\Lambda_{1,w}} =\int_0^\infty \mu_t(f)  w(t)\ dt <\infty.
\]
The space $(\Lambda_{1,w}, \|\cdot\|_{\Lambda_{1,w}})$ is a  fully symmetric Banach function space. According to \cite{Kaminska-Maligranda-Persson}, $\Lambda_{1,w}$ is $r$-concave if and only  if 
(for $1/r +1/r'=1$), 
\[
\Big( \frac{1}{t} \int_0^t w(s)^{r'} \ ds \Big)^{1/r'} \lesssim w(t), \quad t>0.
\]
With the preceding criterion, one can isolate the family of  weights $w$ for which $\Lambda_{1,w} \in {\rm Int}(L_1, L_q)$ for some $1<q<2$ and therefore the Davis decomposition applies to  martingales from the corresponding Hardy spaces $\H_{\Lambda_{1,w}}^c(\M)$.

\bigskip 

 \noindent{\bf Acknowledgments.} I am grateful to  the  anonymous referee  for a careful reading of the paper and  for  useful comments  which improved the presentation of the paper.




\begin{thebibliography}{10}

\bibitem{Bekjan-Chen-Perrin-Y}
T.~Bekjan, {Z.} Chen, {M.} Perrin, and Z.~Yin, \emph{Atomic decomposition and
  interpolation for {H}ardy spaces of noncommutative martingales}, J. Funct.
  Anal. \textbf{258} (2010), no.~7, 2483--2505. \MR{2584751 (2011d:46131)}

\bibitem{BENSHA}
C.~Bennett and R.~Sharpley, \emph{Interpolation of operators}, Academic Press
  Inc., Boston, MA, 1988. \MR{89e:46001}

\bibitem{BL}
J.~Bergh and J.~L{\"o}fstr{\"o}m, \emph{Interpolation spaces. {A}n
  introduction}, Springer-Verlag, Berlin, 1976, Grundlehren der Mathematischen
  Wissenschaften, No. 223. \MR{MR0482275 (58 \#2349)}

\bibitem{Cadilhac2}
L.~Cadilhac, \emph{Noncommutative {K}hintchine inequalities in interpolation
  spaces of {$L_p$}-spaces}, Adv. Math. \textbf{352} (2019), 265--296.
  \MR{3961739}

\bibitem{Cadilhac1}
\bysame, \emph{Majorization, interpolation and noncommutative {K}hinchin
  inequalities}, Studia Math. \textbf{258} (2021), no.~1, 1--26. \MR{4214351}

\bibitem{Cad-Ricard}
L.~Cadilhac and E.~Ricard, \emph{Sums of free variables in fully symmetric
  spaces}, arXiv: 1911.06180v3 [math.OA].

\bibitem{Chen-Ran-Xu}
Z.~Chen, N.~Randrianantoanina, and Q.~Xu, \emph{Atomic decompositions for
  noncommutative martingales}, arXiv:2001.08775v1 [math.OA].

\bibitem{Cuc}
I.~Cuculescu, \emph{Martingales on von {N}eumann algebras}, J. Multivariate
  Anal. \textbf{1} (1971), 17--27. \MR{45 \#4464}

\bibitem{Cw-Mil-Sag}
M.~Cwikel, M.~Milman, and Y.~Sagher, \emph{Complex interpolation of some
  quasi-{B}anach spaces}, J. Funct. Anal. \textbf{65} (1986), no.~3, 339--347.
  \MR{826431}

\bibitem{Dirksen-Thesis}
S.~Dirksen, \emph{Noncommutative and vector-valued {R}osenthal inequalities},
  Ph.D. Thesis dissertation (2011).

\bibitem{Dirksen2}
\bysame, \emph{Noncommutative {B}oyd interpolation theorems}, Trans. Amer.
  Math. Soc. \textbf{367} (2015), no.~6, 4079--4110. \MR{3324921}

\bibitem{FK}
T.~Fack and H.~Kosaki, \emph{Generalized $s$-numbers of $\tau$-measurable
  operators}, Pacific J. Math. \textbf{123} (1986), 269--300. \MR{87h:46122}

\bibitem{Fefferman-Stein}
C.~Fefferman and E.~M. Stein, \emph{{$H\sp{p}$} spaces of several variables},
  Acta Math. \textbf{129} (1972), no.~3-4, 137--193. \MR{447953}

\bibitem{GA}
{A. M.} Garsia, \emph{Martingale inequalities: {S}eminar notes on recent
  progress}, W. A. Benjamin, Inc., Reading, Mass.-London-Amsterdam, 1973,
  Mathematics Lecture Notes Series. \MR{56 \#6844}

\bibitem{Hitc-SMS}
P.~Hitczenko and S.~Montgomery-Smith, \emph{Tangent sequences in {O}rlicz and
  rearrangement invariant spaces}, Math. Proc. Cambridge Philos. Soc.
  \textbf{119} (1996), no.~1, 91--101. \MR{1356161}

\bibitem{Holm}
T.~Holmstedt, \emph{Interpolation of quasi-normed spaces}, Math. Scand.
  \textbf{26} (1970), 177--199. \MR{54 \#3440}

\bibitem{Hong-Junge-Parcet}
G.~Hong, M.~Junge, and J.~Parcet, \emph{Algebraic {D}avis decomposition and
  asymmetric {D}oob inequalities}, Comm. Math. Phys. \textbf{346} (2016),
  no.~3, 995--1019. \MR{3537343}

\bibitem{Jason-Jones}
S.~Janson and P.~W. Jones, \emph{Interpolation between {$H\sp{p}$} spaces: the
  complex method}, J. Funct. Anal. \textbf{48} (1982), no.~1, 58--80.
  \MR{671315}

\bibitem{Jiao2}
Y.~Jiao, \emph{Martingale inequalities in noncommutative symmetric spaces},
  Arch. Math. (Basel) \textbf{98} (2012), no.~1, 87--97. \MR{2885535}

\bibitem{Jiao-Ran-Wu-Zhou}
Y.~Jiao, N.~Randrianantoanina, L.~Wu, and D.~Zhou, \emph{Square functions for
  noncommutative differentially subordinate martingales}, Comm. Math. Phys.
  \textbf{374} (2020), no.~2, 975--1019. \MR{4072235}

\bibitem{Jiao-Sukochev-Zanin}
Y.~Jiao, F.~Sukochev, and D.~Zanin, \emph{Johnson-{S}chechtman and {K}hintchine
  inequalities in noncommutative probability theory}, J. Lond. Math. Soc. (2)
  \textbf{94} (2016), no.~1, 113--140. \MR{3532166}

\bibitem{Jiao-Sukochev-Zanin-Zhou}
Y.~Jiao, F.~Sukochev, D.~Zanin, and D.~Zhou, \emph{Johnson--{S}chechtman
  inequalities for noncommutative martingales}, J. Funct. Anal. \textbf{272}
  (2017), no.~3, 976--1016. \MR{3579131}

\bibitem{Jones}
P.~W. Jones, \emph{On interpolation between {$H\sp{1}$} and {$H\sp{\infty }$}},
  Interpolation spaces and allied topics in analysis ({L}und, 1983), Lecture
  Notes in Math., vol. 1070, Springer, Berlin, 1984, pp.~143--151. \MR{760480}

\bibitem{Ju}
M.~Junge, \emph{Doob's inequality for non-commutative martingales}, J. Reine
  Angew. Math. \textbf{549} (2002), 149--190. \MR{2003k:46097}

\bibitem{Junge-Perrin}
M.~Junge and M.~Perrin, \emph{Theory of {$\mathcal H\sb p$}-spaces for
  continuous filtrations in von {N}eumann algebras}, Ast\'erisque (2014),
  no.~362, vi+134. \MR{3241706}

\bibitem{JX}
M.~Junge and Q.~Xu, \emph{Noncommutative {B}urkholder/{R}osenthal
  inequalities}, Ann. Probab. \textbf{31} (2003), no.~2, 948--995.
  \MR{2004f:46078}

\bibitem{KaltonSMS}
N.~Kalton and S.~Montgomery-Smith, \emph{Interpolation of {B}anach spaces},
  Handbook of the geometry of Banach spaces, Vol.\ 2, North-Holland, Amsterdam,
  2003, pp.~1131--1175. \MR{1 999 193}

\bibitem{Kalton-Sukochev}
N.~J. Kalton and F.~A. Sukochev, \emph{Symmetric norms and spaces of
  operators}, J. Reine Angew. Math. \textbf{621} (2008), 81--121. \MR{2431251
  (2009i:46118)}

\bibitem{Kaminska-Maligranda-Persson}
A.~Kami\'{n}ska, L.~Maligranda, and L.~E. Persson, \emph{Convexity, concavity,
  type and cotype of {L}orentz spaces}, Indag. Math. (N.S.) \textbf{9} (1998),
  no.~3, 367--382. \MR{1692169}

\bibitem{KPS}
S.~G. Kre\u{\i}n, Yu.~\={I}. Petun\={\i}n, and E.~M. Sem\"{e}nov,
  \emph{Interpolation of linear operators}, Translations of Mathematical
  Monographs, vol.~54, American Mathematical Society, Providence, R.I., 1982,
  Translated from the Russian by J. SzH{u}cs. \MR{649411}

\bibitem{Maligranda2}
L.~Maligranda, \emph{Orlicz spaces and interpolation}, Semin\'arios de
  Matem\'atica [Seminars in Mathematics], vol.~5, Universidade Estadual de
  Campinas, Departamento de Matem\'atica, Campinas, 1989. \MR{2264389
  (2007e:46025)}

\bibitem{Musat-inter}
M.~Musat, \emph{Interpolation between non-commutative {BMO} and non-commutative
  {$L_p$}-spaces}, J. Funct. Anal. \textbf{202} (2003), no.~1, 195--225.
  \MR{1994770}

\bibitem{Perrin2}
M.~Perrin, \emph{In\'egalit\'es de martingales non commutatives et
  applications}, Ph.D. Thesis dissertation (2011).

\bibitem{Perrin}
\bysame, \emph{A noncommutative {D}avis' decomposition for martingales}, J.
  Lond. Math. Soc. (2) \textbf{80} (2009), no.~3, 627--648. \MR{2559120
  (2011e:46104)}

\bibitem{Pisier-int}
G.~Pisier, \emph{Interpolation between {$H^p$} spaces and noncommutative
  generalizations. {I}}, Pacific J. Math. \textbf{155} (1992), no.~2, 341--368.
  \MR{1178030}

\bibitem{Pisier-Martingale}
\bysame, \emph{Martingales in {B}anach spaces}, Cambridge Studies in Advanced
  Mathematics, vol. 155, Cambridge University Press, Cambridge, 2016.
  \MR{3617459}

\bibitem{PX}
G.~Pisier and Q.~Xu, \emph{Non-commutative martingale inequalities}, Comm.
  Math. Phys. \textbf{189} (1997), 667--698. \MR{98m:46079}

\bibitem{PX3}
\bysame, \emph{Non-commutative {$L\sp p$}-spaces}, Handbook of the geometry of
  Banach spaces, Vol.\ 2, North-Holland, Amsterdam, 2003, pp.~1459--1517.
  \MR{2004i:46095}

\bibitem{Qiu1}
Y.~Qiu, \emph{A non-commutative version of {L}\'epingle-{Y}or martingale
  inequality}, Statist. Probab. Lett. \textbf{91} (2014), 52--54. \MR{3208115}

\bibitem{Ran16}
N.~Randrianantoanina, \emph{Square function inequalities for non-commutative
  martingales}, Israel J. Math. \textbf{140} (2004), 333--365. \MR{2005c:46091}

\bibitem{Ran21}
\bysame, \emph{Conditioned square functions for noncommutative martingales},
  Ann. Probab. \textbf{35} (2007), no.~3, 1039--1070. \MR{2319715
  (2009d:46112)}

\bibitem{RW}
N.~Randrianantoanina and L.~Wu, \emph{Martingale inequalities in noncommutative
  symmetric spaces}, J. Funct. Anal. \textbf{269} (2015), no.~7, 2222--2253.
  \MR{3378874}

\bibitem{Ran-Wu-Xu}
N.~Randrianantoanina, L.~Wu, and Q.~Xu, \emph{Noncommutative {D}avis type
  decompositions and applications}, J. Lond. Math. Soc. (2) \textbf{99} (2019),
  no.~1, 97--126.

\bibitem{Ran-Wu-Zhou}
N.~Randrianantoanina, L.~Wu, and D.~Zhou, \emph{Atomic decompositions and
  asymmetric {D}oob inequalities in noncommutative symmetric spaces}, J. Funct.
  Anal. \textbf{280} (2021), no.~1, 64 pp. \MR{4157678}

\bibitem{Turpin1}
P.~Turpin, \emph{Convexit\'{e}s dans les espaces vectoriels topologiques
  g\'{e}n\'{e}raux}, Dissertationes Math. (Rozprawy Mat.) \textbf{131} (1976),
  221. \MR{423044}

\bibitem{Weisz3}
F.~Weisz, \emph{Interpolation between martingale {H}ardy and {BMO} spaces, the
  real method}, Bull. Sci. Math. \textbf{116} (1992), no.~2, 145--158.
  \MR{1168308}

\bibitem{Weisz}
\bysame, \emph{Martingale {H}ardy spaces and their applications in {F}ourier
  analysis}, Lecture Notes in Mathematics, vol. 1568, Springer-Verlag, Berlin,
  1994. \MR{MR1320508 (96m:60108)}

\bibitem{Wolff}
T.~H. Wolff, \emph{A note on interpolation spaces}, Harmonic analysis
  ({M}inneapolis, {M}inn., 1981), Lecture Notes in Math., vol. 908, Springer,
  Berlin-New York, 1982, pp.~199--204. \MR{654187}

\bibitem{Xu-inter}
Q.~Xu, \emph{Applications du th\'{e}or\`eme de factorisation pour des fonctions
  \`a valeurs op\'{e}rateurs}, Studia Math. \textbf{95} (1990), no.~3,
  273--292. \MR{1060730}

\bibitem{X}
\bysame, \emph{Analytic functions with values in lattices and symmetric spaces
  of measurable operators}, Math. Proc. Cambridge Philos. Soc. \textbf{109}
  (1991), 541--563. \MR{92g:46036}

\end{thebibliography}

\def\cprime{$'$}
\providecommand{\bysame}{\leavevmode\hbox to3em{\hrulefill}\thinspace}
\providecommand{\MR}{\relax\ifhmode\unskip\space\fi MR }
\providecommand{\MRhref}[2]{%
  \href{http://www.ams.org/mathscinet-getitem?mr=#1}{#2}
}
\providecommand{\href}[2]{#2}

\end{document}